\documentclass[10pt, reqno]{amsart}
\usepackage{hyperref}
\usepackage{amsmath}
\usepackage{amscd}
\usepackage{amsthm}
\usepackage{amssymb}
\usepackage{latexsym}
\usepackage{eufrak}
\usepackage{euscript}
\usepackage{epsfig}
\usepackage{graphics}
\usepackage{array}
\usepackage{enumerate}
\usepackage{dsfont}
\usepackage{color}
\usepackage{wasysym}
\usepackage{graphicx}
\usepackage{amsfonts}
\usepackage{mathrsfs}
\usepackage{dsfont}
\usepackage{indentfirst}
\usepackage{esint}

\newcommand{\R}{{\mathbb R}}

\newcommand{\Q}{{\mathbb Q}}
\newcommand{\N}{{\mathbb N}}

\newcommand{\al}{{\alpha}}
\newcommand{\be}{{\beta}}

\newcommand{\G}{\mathcal{G}}
\newcommand{\RG}{\mathcal{RG}}
\newcommand{\RRG}{\mathcal{RRG}}
\newcommand{\FG}{\mathcal{FG}}
\newcommand{\diam}{\mathrm{diam}}

\newcommand{\ex}{\mathrm{Ext}}
\newcommand{\SM}{\mathrm{SM}}

\newtheorem{thm}{Theorem}[section]
\newtheorem{coro}[thm]{Corollary}
\newtheorem{lemma}[thm]{Lemma}
\newtheorem{pro}[thm]{Proposition}

\theoremstyle{definition}
\newtheorem{definition}[thm]{Definition}
\newtheorem{example}[thm]{Example}
\newtheorem{remark}[thm]{Remark}

\newtheorem{fact}[thm]{Fact}

\newcommand{\Hmm}[1]{\leavevmode{\marginpar{\tiny%
$\hbox to 0mm{\hspace*{-0.5mm}$\leftarrow$\hss}%
\vcenter{\vrule depth 0.1mm height 0.1mm width \the\marginparwidth}%
\hbox to 0mm{\hss$\rightarrow$\hspace*{-0.5mm}}$\\\relax\raggedright
#1}}}

\begin{document}

\title{Spectral distances on graphs}

\author{Jiao Gu}
\address{Jiao Gu, School of Science, Jiangnan University, 214122, Wuxi, China; Mathematics and Science college, Shanghai Normal University, 200234, Shanghai, China; Max Planck Institute for Mathematics in the Sciences,
04103, Leipzig, Germany.} \email{jiaogu@mis.mpg.de}

\author{Bobo Hua}
\address{Bobo Hua, Max Planck Institute for Mathematics in the Sciences,
04103, Leipzig, Germany; School of Mathematical Sciences, LMNS, Fudan University, Shanghai
200433, China.} \email{bobohua@mis.mpg.de}

\author{Shiping Liu}
\address{Shiping Liu, Max Planck Institute for Mathematics in the Sciences,
04103, Leipzig, Germany; Department of Mathematical Sciences, Durham University, DH1 3LE Durham, United Kingdom.} \email{shiping.liu@durham.ac.uk}

\begin{abstract}
By assigning a probability measure via the spectrum of the normalized Laplacian to each graph and
%For each graph, we associate the spectrum of the normalized
%Laplacian with a probability measure, called the spectral measure.
using $L^p$ Wasserstein distances between probability measures, we
define the corresponding spectral distances $d_p$ on the set of all
graphs. This approach can even be extended to measuring the distances between infinite graphs. We prove that the diameter of the set of graphs,
as a pseudo-metric space equipped with $d_1$,
%$L^1$ Wasserstein distance
is one. We further study the behavior of $d_1$ when the size of graphs tends to infinity by interlacing inequalities aiming at exploring large real networks.
A monotonic relation between $d_1$ and the evolutionary distance of biological networks is observed in simulations.
%Furthermore, by the interlacing results of the spectra we study
%the convergence of graphs under the spectral distances. We also
%explore some applications in biology.


\end{abstract}
\maketitle

\section{Introduction}
One major interest in graph theory is to explore the differences of graphs in structure, that is, in the sense of graph isomorphism.
%However, this problem has been proven to be an NP problem, of which the most notable characteristic is that no fast solution to it is known.
In computational complexity theory, the subgraph isomorphism problem, like many combinational problems in graph theory, is NP hard. Therefore,
a method that gives a quick and easy estimate of the difference between two graphs is desirable \cite{macindoe2010graph}.
%even some certain level of approximation in the structural comparison of graphs should be accepted for the sake of tractability \cite{macindoe2010graph}.
%In computational complexity theory, the graph isomorphism problem belongs to NP (neither known to be P nor NP-complete) and the subgraph isomorphism problem belongs to NP-complete.
%Therefore, we have to accept certain level of approximation in the structural comparison of graphs for the sake of tractability \cite{macindoe2010graph}.
As we know, all the topological information of one graph can be found in its adjacency matrix. The spectral graph theory studies the relationship between the properties of graphs and the spectra of their representing matrices, such as adjacency matrices and Laplace matrices \cite{biggs1994algebraic, cvetkovic1980spectra, chung1997spectral}. In particular, some important topological information of a graph can be extracted from its specific eigenvalue like the first or the largest one, see e.g. \cite{cvetkovic1980spectra, chung1997spectral, Trevisan2012, BJ, JL11, BJL12, BHJ12}. The approach of reading information from the entire spectrum of a graph was explored in \cite{BanJ08, BanJLAA08, BanJ09, LOT13, Liu13} etc.
In spite of the existence of co-spectral graphs (see \cite[Chapter 3]{Thune12} for a general construction and the references therein), the spectra of graphs can support us one way on exploring problems that involve (sub-)graph isomorphism by the fast computation algorithms and the close relationship with the structure of graphs.

A spectral distance on the set of finite graphs of the same size, i.e. the same number of vertices, was suggested in a problem of Richard Brualdi in \cite{StevanovicProblems2007} to explore the so-called cospectrality of a graph. It was further studied in \cite{JSspectraldistance2012} using the spectra of adjacency matrices.
Employing certain Gaussian measures associated to the spectra of normalized Laplacians and the corresponding $L^1$ distances, the first named author, Jost, the third named author and Stadler \cite{GJS2014,Gu2014} explored a spectral distance well-defined on the set of all finite graphs without any constraint about sizes.
In this paper, instead of the Gaussian measures, we assign Dirac measures to graphs through the spectra of normalized Laplacians and use the Wasserstein distances between probability measures to propose spectral distances between graphs.
In fact, this notion of spectral distances provides a metrization of the notion of spectral classes of graphs introduced in \cite{GJS2014} via the weak convergence of the corresponding Dirac measures. The Spectral class can be considered as a weak notion of graph limits (see the concepts of graphon, graphing and related theories in the monograph of Lov\'{a}sz \cite{Lovasz}).
This notion of spectral distances is even adaptable for weighted infinite graphs. And we can prove diameter estimates with respect to these distances, which are sharp for certain cases.

A weighted graph $G$ is a triple $(V,E,\theta)$ where $V$ is
the set of vertices, $E$ is the set of edges and
$\theta:E\to(0,\infty)$, $(x,y)\mapsto\theta_{xy},$ is the (symmetric) edge
weight function. We write $x\sim y$ or $xy\in E$ if $\theta_{xy}>0$. We assume that for any vertex $x,$ the weighted degree defined by $\theta_x:=\sum_{y\sim x}\theta_{xy}$ is finite and $\theta_{xx}=0$ (i.e. there is no self-loops).

Let us first consider finite weighted graphs. The normalized Laplacian of $G=(V,E,\theta)$ is defined as, for any function $f: V\rightarrow\R$ and
any $x\in V$,
\begin{equation}\label{e:laplaceFinite}
\Delta_Gf(x)=f(x)-\frac{1}{\theta_x}\sum_{y\sim x}f(y)\theta_{xy}.
\end{equation}
This operator can be extended to an infinite weighted graph which has countable vertex set $V$ but is not necessarily locally finite (see \cite{KellerLenz12} or Section \ref{s:Pre} below). As a matrix, $\Delta_G$ is unitarily equivalent to the Laplace matrix studied in \cite{chung1997spectral}.

If $x\in V$ is an isolated vertex, i.e. $\theta_x=0$, (\ref{e:laplaceFinite}) reads as $\Delta_Gf(x)=f(x)$.
This implies that an isolated vertex contribute an eigenvalue $1$ to the spectrum of $\Delta_G$, denoted by $\sigma(G)$.
%We denote by $\ell^2(V,\theta)$ the space of $\ell^2$-integrable functions on $V$ with respect to the measure $\theta$. For details, we refer to Section \ref{s:Pre}.
%
%As a bounded linear operator on $\ell^2(V,\theta),$ the normalized Laplacian of a (possibly infinite) weighted graph $G=(V,E,\theta)$ is densely defined as, for any finitely supported function $f:V\to \R,$
%$$\Delta_Gf(x)=f(x)-\frac{1}{\theta_x}\sum_{y\sim
%x}f(y)\theta_{xy}.$$ It is well-known that its spectrum,
%denoted by $\sigma(G)$, is contained in $[0,2].$ We shall assign every finite weighted graph a probability measure (for infinite one, a set of probability measures) by its spectrum information.
%
%Let us first consider finite weighted graphs.
%%For a finite graph
%%with $K$ isolated vertices,
%%$G=G_1\cup\{x_i\}_{i=1}^K$,
%We take the convention that the spectrum of the graph of a single vertex is $\{1\}$.
%%$\sigma(G)=\sigma(G_1)\cup\{\underbrace{1,\cdots,1}_{K}\}.$
%Note that
%this is different from the standard one in the literature
%that the isolated vertices contribute to the spectrum zero eigenvalues.
In this way, by the absence of the self-loops, the
spectrum of any finite weighted graph
$\sigma(G)=\{\lambda_i\}_{i=1}^N,$ counting the multiplicity, satisfies the trace condition
\begin{equation}\label{e:trace condition(introduction)}\sum_{i=1}^N\lambda_i=N\end{equation} where $N=|V|.$ It is well-known that $\sigma(G)$ is contained in $[0,2]$. We associate
to $\sigma(G)$ a probability measure on $[0,2]$ as
follows:
\begin{equation}\label{e:spectralmeasure}\mu_{\sigma(G)}:=\frac1N\sum_i\delta_{\lambda_i},\end{equation} where $\delta_{\lambda_i}$ is the Dirac measure concentrated on $\lambda_i.$
We call $\mu_{\sigma(G)}$ the \emph{spectral measure} for a finite weighted
graph. (This is known as the empirical distribution of the eigenvalues in random matrix theory.) Denote by $P([0,2])$ the set of probability measures on the
interval $[0,2]$. For any $\mu\in P([0,2]),$ the first moment of $\mu$ is
defined as
$m_1(\mu):=\int_{[0,2]}\lambda\ d\mu(\lambda).$
The trace
condition \eqref{e:trace condition(introduction)} is then translated to
\begin{equation}
m_1(\mu_{\sigma(G)})=1.
\end{equation}
This is a key property of the spectral measures for our further investigations.

Let $d_p^W$ ($1\leq p<\infty$) be the $p$-th Wasserstein
distance on $P([0,2])$. That is, for any $\mu$, $\nu\in P([0,2])$ (see e.g. \cite{Villani09}),
$$d_{p}^W (\mu, \nu):=\left( \inf_{\pi \in \Pi (\mu,
\nu)} \int_{[0,2] \times [0,2]} d(x, y)^{p}d\pi (x, y) \right)^{1/p},$$
where $\Pi (\mu, \nu)$ denotes the collection of all measures on
$[0,2]\times [0,2]$ with marginals $\mu$ and $\nu$ on the first and second
factors respectively, i.e. $\pi\in \Pi(\mu,\nu)$ if and only if
$\pi(A\times [0,2])=\mu(A)$ and $\pi([0,2]\times B)=\nu(B)$ for all Borel
subsets $A,B\subseteq [0,2]$.

It is well-known that $(P([0,2]), d_p^W)$ is a complete metric space
for $p\in [1,\infty)$ which induce the weak topology of measures in $P([0,2])$(see e.g. \cite[Theorem~6.9]{Villani09}).

One can prove that $\diam
(P([0,2]),d_p^W)=2$. Indeed, on one hand, for any $\mu,\nu\in
P([0,2])$ by the optimal transport interpretation of Wasserstein
distance, $d_p^W(\mu,\nu)\leq 2.$ On the other hand,
$d_p^W(\delta_0,\delta_2)=2$. (Recall that $\delta_0, \delta_2$ are the Dirac measures concentrated on $0, 2$, respectively.)

\begin{definition}\label{d:spectal distance}Given two finite weighted graphs $G=(V,E,\theta)$ and
$G'=(V',E',\theta'),$ the \emph{spectral distance} between $G$
and $G'$ is defined as
\begin{equation}\label{e:defSpectDist}d_p(G, G'):=d_p^W(\mu_{\sigma(G)},\mu_{\sigma(G')}).\end{equation}\end{definition}

We denote by $\mathcal{FG}$ the space of all finite weighted graphs.
Then for
any $1\leq p<\infty,$ $(\mathcal{FG},d_p)$ is a pseudo-metric space.
This is not a metric space due to the existence of co-spectral graphs.
However, in applications this spectral consideration leads to the simplification of
measuring the discrepancy of graphs.

One of the main results of our paper is the following theorem.
\begin{thm}\label{t:diam estimate}
For any $1\leq p<\infty,$ we have
$$\diam(\FG,d_p)\leq 2^{1-\frac{1}{p}}.$$
\end{thm}
\begin{remark}
\

\begin{enumerate}[(a)]
\item Embedded as a subspace of $P([0,2]),$ $\FG$ is a proper subspace by considering the diameters.
\item One can prove an upper bound directly by using Chebyshev inequality, see
Theorem \ref{t:Chebyshev diam estimate}. Clearly, this theorem
improves that estimate.
\item This estimate is tight for $p=1,$ i.e. $\diam(\FG,d_1)=1,$ see Corollary \ref{c:sharpness}.
\item We don't claim the sharpness of upper bound estimates for $p\in(1,\infty).$
\end{enumerate}
\end{remark}

In fact, Theorem \ref{t:diam estimate} follows from the estimates on the
Wasserstein distance of probability measures in condition of the first moments.
\begin{thm}[Measure-theoretic version]\label{t:measure version}
For any $\mu,\nu\in P([0,2])$ with $m_1(\mu)=m_1(\nu)=1$ and $p\in
[1,\infty),$
\begin{equation}
d^W_p(\mu,\nu)\leq 2^{1-\frac{1}{p}}.
\end{equation}
\end{thm}

By Proposition \ref{p:cumulative functions} and Lemma \ref{l:cumulative dis} below, one easily
shows that the above measure-theoretic estimate is equivalent to the
following analytic estimate.
\begin{thm}[Analytic version]\label{t:analytic version}
Let $f,g:[0,1]\to[0,2]$ be two nondecreasing functions such that
$\int_0^1 f(x)dx=\int_0^1 g(x)dx=1.$ Then for any $p\in [1,\infty)$
\begin{equation}
\left(\int_0^1|f-g|^p(x)dx\right)^{\frac{1}{p}}\leq 2^{1-\frac{1}{p}}.
\end{equation}
\end{thm}
Section \ref{s:maintheorem} is devoted to the proofs of Theorem \ref{t:diam estimate}, \ref{t:measure version} and Theorem
\ref{t:analytic version}.

We extend our approach of the spectral distance to infinite graphs (with countable vertex set V) in Section
\ref{s:spectral distances of infinite graphs}. Note that in the above arguments we only use
the normalization of the first moment of the spectral measures, i.e.
$m_1(\mu_{\sigma(G)})=1,$ our results generalize to all weighted
graphs including the infinite ones. For spectral measures with distinguished vertices on infinite
graphs, we refer to Mohar-Woess \cite{Mohar1989}. We introduce two
definitions of spectral measures for infinite graphs. One is defined
via the exhaustion of the infinite graphs by the spectral measures
of normalized Dirichlet Laplacians on subgraphs. The other is
defined for random rooted graphs following Benjamini-Schramm
\cite{Benjamini2001a}, Aldous-Lyons \cite{Aldous2007} and Ab\'ert-Thom-Vir\'ag \cite{ATV}.

We denote by $\G$ the collection of all (possibly infinite) weighted
graphs. For any $G\in \G,$ we define $\SM(G)$ as the spectral
measures of $G$ by exhaustion, see Definition \ref{d:infinite by
exhaustion}, which is a closed subset of $P([0,2])$. Then $\G$
endowed with the Hausdorff distance induced from the metric space
$(P([0,2]),d_p^W)$, denoted by $d_{p,H},$ is a pseudo-metric space.
A direct application of Theorem \ref{t:measure version} yields the
following corollary (recalled below as Theorem \ref{T:diameterInfi1}).
\begin{coro} For $p\in [1,\infty),$
$$\diam(\G,d_{p,H})\leq 2^{1-\frac{1}{p}}.$$
\end{coro}

For any $D\geq 1,$ we denote by $\RRG_D$ the collection of random
rooted graphs of degree $D,$ see Section \ref{s:random rooted
graphs} for definitions. Any finite weighted graph $G$ gives rise to
a random graph by assigning the root of $G$ uniformly randomly.
There are many interesting class of random rooted graphs such as
unimodular and sofic ones, see \cite{ATV}. For each random rooted graph
$G\in \RRG_D,$ we associate it with an expected spectral measure,
denoted by $\mu_G.$ In this way, $\RRG_D$ endowed with $d_p^W$
Wasserstein distance for expected spectral measures ($d_p$ in short)
is a pseudo-metric space. By Theorem \ref{t:measure version}, one can
prove the following corollary (recalled below as Theorem \ref{T:diameterInfi2}).
\begin{coro}
For $p\in [1,\infty),$
$$\diam (\RRG_D,d_p)\leq 2^{1-\frac{1}{p}}.$$
\end{coro}

In
fact, there are examples of finite graphs which saturate the upper bounds for $p=1$, see Example \ref{example:distance equals one} and \ref{example:distance approach one}.
\begin{coro}\label{c:sharpness}All upper bounds on $d_1$ are tight, i.e.
$$\diam(\FG,d_1)=\diam(\G,d_{1,H})=\diam (\RRG_D,d_1)=1.$$
\end{coro}

We then concentrate on the spectral distance $d_1$. In Section \ref{s:examples}, we calculate $d_1$ on several particular classes of graphs.
For our purpose of application to large real networks, we are more concerned with the behavior of $d_1$ when the size of graphs $N$ tends to infinity. We observe convergence behaviors like $\mathcal{O}(1/N^2)$, $\mathcal{O}(1/N)$ in those examples.

The asymptotic behavior of $d_1$ is studied in general in Section \ref{s:Interlacing} by employing interlacing inequalities of the spectra of finite weighted graphs. For two graphs $G$ and $G'$, which differ from each other by some standard operations including e.g. edge deleting, vertex replication, vertices contraction and edge contraction, we prove
\begin{equation}
 d_1(G, G')\leq \frac{C}{N},
\end{equation}
where $C$ depends only on the operations and is independent of the size $N$ of $G$ (see Theorem \ref{t:Asymptotic behavior}).
By this result, we further derive a convergence result of graphs under the $d_1$ distance.

In the last section, we apply the distance $d_1$ to study the evolutionary process of biological networks by simulations. We start from a Barab\'{a}si-Albert scale-free network, which has proven to be a very common type of real large networks \cite{BA1999Science}. We then simulate the evolutionary process by the operations, edge-rewiring and duplication-divergence respectively. We observe a monotonic relation between $d_1$ and the evolutionary distance, which is a crucial point to anticipate further applications in exploring evolutionary history of biological networks.

\section{Preliminaries, spectral measures and spectral distances}\label{s:Pre}
In this section, we recall basics about graph spectra and Wasserstein distances
on the space of probability measures, and define the spectral
distances of finite graphs. The spectral distances of infinite
graphs and random graphs will be postponed to Section
\ref{s:spectral distances of infinite graphs}.

Let us consider a possibly infinite weighted graph $G=(V, E, \theta)$, where $V$ is a countable (possible infinite) set. We require that the weight function $\theta$ satisfies
\begin{align*}
    \sum_{y\in V}\theta_{xy}<\infty,\qquad \forall x\in V.
\end{align*} The weighted degree of the
vertex $x\in V$ is still defined as $\theta_x:=\sum_{y\sim x}\theta_{xy}$. The graph is called connected if for every two vertices $x,y\in V$ there exists
a finite path $x=x_{0}\sim x_{1}\sim\cdots\sim x_{n}=y$ connecting $x$ and $y$.

We define the \emph{(formal)
normalized Laplacian} $\Delta$ on the \emph{formal domain}
\begin{align*}
    F(V):=\{f:V\to\R\mid \sum_{y\in V}\theta_{xy}|f(y)|<\infty\mbox{ for all }x\in V\},
\end{align*}
by
\begin{align*}
    \Delta f(x)=\frac{1}{\theta_x}\sum_{y\in V}\theta_{xy}(f(x)-f(y)).
\end{align*}

As a linear operator, its restriction to the Hilbert space $\ell^{2}(V,\theta):=\{f:V\to\R| \sum_{x\in
V}|f(x)|^2\theta_x<\infty\},$ denoted by $\Delta_G$, coincides with the generator of the Dirichlet form
$$Q(f)=\frac{1}{2}\sum_{x,y\in V}\theta_{xy}|f(x)-f(y)|^{2},$$ defined on $\ell^2(V,\theta),$ for details
see \cite{KellerLenz12}.

If $G=(V,E,\theta)$ is a weighted graph without
isolated vertices, i.e. $\theta_x>0$ for all $x\in V$, then the
normalized Laplacian of $G$ can be rephrased as
$$\Delta_G:=I-D^{-1}A,$$ where $D$ is the degree operator and $A$ is the adjacency operator (defined as $D\tau_x=\theta_x\tau_x$ and $A\tau_x=\sum_{y\sim x}\theta_{yx}\tau_y$, where $\tau_x(y)=1$ if $y=x$ and $0$ otherwise), i.e. for any finitely supported function
$f: V\to \R,$ $$\Delta_Gf(x)=f(x)-\frac{1}{\theta_x}\sum_{y\sim
x}f(y)\theta_{xy}.$$ Since $D^{-1}A$ is a bounded selfadjoint
operator with operator norm less than or equal to $1$ on
$\ell^2(V,\theta)$, the spectrum of $\Delta_G,$ denoted by $\sigma(G)$,
is contained in the interval $[0,2].$

%In this section, we consider the spectral properties of normalized
%Laplacians of finite graphs. For finite weighted graphs without
%isolated vertices, the normalized Laplacians are defined as above.
%For finite weighted graphs with isolated vertices, we take the following convention.
%Assume $G=G_1\cup\{x_i\}_{i=1}^K$ is a finite graph with $K$ isolated
%vertices, then we set
%$$\sigma(G)=\sigma(G_1)\cup\{\underbrace{1,\cdots,1}_{K}\}.$$ As we have commented in the introduction,
%this convention is different from the standard one adopted in the literature. However, the reason of
%taking this convention will be clear soon.
We order the
spectrum of any finite weighted graph $G$ in the
nondecreasing way:
$$0\leq\lambda_1\leq \cdots\leq \lambda_N\leq 2,$$ where $N=|V|.$
For convenience, we also denote the spectrum of $G$ by a vector,
called spectral vector of $G$,
$\lambda_G:=(\lambda_i)_{i=1}^N=(\lambda_1,\lambda_2,\cdots,\lambda_N)\in
[0,2]^N.$

\subsection{Spectral measures}
Let $G$ be a finite weighted graph. We denote by $F_G$ the cumulative distribution function associated to $\mu_{\sigma(G)}$ (recall (\ref{e:spectralmeasure})), and by
$$F^{-1}_{G}(x):=\inf\{t\in \R: F_{G}(t)>x\}$$
the inverse cumulative distribution function.
Since $\sigma(G)\subseteq [0,2],$ we have $F_{G}:[0,2]\to[0,1]$ and $F^{-1}_{G}:[0,1]\to[0,2]$.
Recalling the trace condition (\ref{e:trace condition(introduction)}), we have the following
proposition.
\begin{pro}\label{p:cumulative functions}
Let $G=(V,E,\theta)$ be a finite weighted graph. Then the following
are true:
\begin{enumerate}[(a)]
\item $F_{G}$ and $F^{-1}_{G}$ are nonnegative nondecreasing step functions;
\item $\int_0^2 F_{G}(x) dx=1$;
\item
%\label{e:normalization}
$\int_{0}^1F^{-1}_{G}(x)dx=1$.
\end{enumerate}
\end{pro}
\begin{proof}
$(a)$ is trivial. $(c)$ follows from the trace condition
\eqref{e:trace condition(introduction)}. $(b)$ is equivalent to $(c)$ since the
total area of the rectangle $[0,1]\times [0,2]$ is 2.
\end{proof}

\subsection{Spectral distances}
%Let $(X, d)$ be a Polish space (i.e. a complete, separable metric
%space). For $p\in [1,\infty)$, let $P_p(X)$ denote the collection of
%all probability measures $\mu$ on M with finite $p$-th moment: for
%some (hence all) $x_0\in X$, $\int_{X} d(x, x_{0})^{p}d\mu (x) <
%\infty$.
%\begin{definition}The $p$-th Wasserstein distance
%between two probability measures $\mu$ and $\nu$ in $P_p(X)$ is
%defined as
%$$d_{p}^W (\mu, \nu):=\left( \inf_{\pi \in \Pi (\mu,
%\nu)} \int_{X \times X} d(x, y)^{p}d\pi (x, y) \right)^{1/p},$$
%where $\Pi (\mu, \nu)$ denotes the collection of all measures on
%$X\times X$ with marginals $\mu$ and $\nu$ on the first and second
%factors respectively, i.e. $\pi\in \Pi(\mu,\nu)$ if and only if
%$\pi(A\times X)=\mu(A)$ and $\pi(X\times B)=\nu(B)$ for all Borel
%subsets $A,B\subseteq X$.\end{definition}
%
%It is well-known that $(P_p(X),d_p^W)$ are complete metric spaces
%for $p\in [1,\infty)$ which induce the weak topology of measures (see e.g. \cite[Theorem~6.9]{Villani09}).
%
%
%
%
%
%
%\begin{definition}\label{d:spectal distance}Given two finite weighted graphs $G=(V,E,\theta)$ and
%$G'=(V',E',\theta'),$ the \emph{spectral distance} between $G$
%and $G'$ is defined as
%$$d_p(G, G'):=d_p^W(\mu_{\sigma(G)},\mu_{\sigma(G')}).$$\end{definition}
%
%Let us denote by $\mathcal{FG}$ the space of all finite weighted
%graphs. As we have mentioned, for any $1\leq p<\infty,$ $(\mathcal{FG},d_p)$ is a
%pseudo-metric space because of the existence of cospectral
%graphs.

Since the spectrum of the normalized Laplacian of a graph lies in
the interval $[0,2]\subset \R$, one may calculate the spectral
distance (\ref{e:defSpectDist}) explicitly. This is an advantage of
probability measures supported in the $1$-dimensional space. In fact, the spectral
distance between two finite weighted graphs $G$, $G'$, i.e. the Wasserstein distance of two
spectral measures $\mu_{\sigma(G)}, \mu_{\sigma(G')}$, can be calculated by the inverse cumulative distribution functions
$F^{-1}_{G}$ and $F^{-1}_{G'}$ thanks to the following lemma.

\begin{lemma}[see Theorem 8.1 in \cite{Major1978}]\label{l:cumulative dis} Let $\mu,\nu\in P([0,2])$ and $F_{\mu}^{-1},F_{\nu}^{-1}$ be their inverse cumulative distribution functions. Then for any $p\in [1,\infty),$
\begin{equation*}\label{e:inverse cumulative}d_p^W(\mu, \nu)=\left(\int_0^1 |F^{-1}_{\mu}(x)-F^{-1}_{\nu}(x)|^pdx\right)^{1/p}.\end{equation*}
\end{lemma}

One can show that if two graphs having the same number of vertices,
say $N$, then the spectral distance between them is reduced to the
$\ell^p$ distance between the spectral vectors, i.e. for any $1\leq
p<\infty,$
$$d_p(G, G')=\frac1N\|\lambda_{G}-\lambda_{G'}\|_{\ell^p}.$$

In this paper, we are interested in the diameter of the
pseudo-metric space $(\mathcal{FG},d_p)$ for $p\in [1,\infty)$. Recall that we naturally have
$$\diam (\FG,d_p)\leq \diam(P([0,2]), d_p^W)=
2.$$
%This upper bound is less likely to be tight for $p>1$.

We denote by $\{\cdot\}$ a graph consisting of a single vertex
without any edge. Then by our convention, $\sigma(\{\cdot\})=\{1\}.$
Clearly, for any weighted graph $G,$
$$d_p(G,\{\cdot\})\leq 1,\ \ \text{ for } 1\leq p<\infty.$$

In the following, we use (integral) Chebyshev inequality to derive a refined
upper bound for the diameter.
\begin{lemma}[Chebyshev inequality, see {\cite[Section 2.17]{HLP1934}} or \cite{FJ1984}]\label{l:Chebyshev}
For any nonnegative, monotonically increasing integrable functions
$f,g:[0,1]\to [0,\infty),$ we have
\begin{equation}\label{e:Chebyshev}
\int_{0}^1 f(x)g(x)dx\geq \int_0^1f(x)dx\int_0^1g(x)dx.
\end{equation}
\end{lemma}
\begin{thm}\label{t:Chebyshev diam estimate}
For any $1\leq p\leq 2,$ we have
$$\diam(\FG,d_p)\leq \sqrt2,$$ i.e. for any finite weighted graphs
$G$ and $G',$
$$d_p(G, G')\leq \sqrt2.$$
\end{thm}
\begin{proof}
Let us denote $f=F^{-1}_{G}$ and $g=F^{-1}_{G'}.$ Then by
Chebyshev inequality \eqref{e:Chebyshev} and Proposition \ref{p:cumulative functions} (c),
%the normalization \eqref{e:normalization},
$$\int_0^1 fg\geq \int_0^1 f\int_0^1g=1.$$
Hence, for any $1\leq p\leq 2$ we have
\begin{eqnarray*}
\left(\int_{0}^1|f-g|^p\right)^{2/p}&\leq &\int_0^1|f-g|^2=\int_0^1f^2+\int_0^1g^2-2\int_0^1fg\\
&\leq& 2\int_0^1f+2\int_0^1g-2=2,
\end{eqnarray*} where we have used that $f\leq 2$ and $g\leq 2.$
This proves the theorem.
\end{proof}

In the next section, we will give a tighter upper bound for the
diameter estimates. In particular, in the case of $p=1,$ we derive
an optimal upper bound, that is, we will prove that
$\diam(\FG,d_1)=1.$ The tightness of this estimate can be seen from the following two examples.

\begin{example}\label{example:distance equals one}
Let $G=\{\cdot\}$ and $G'=P_2$ be the path on two vertices. Then $\sigma(G')=\{0,2\}$. Hence we have
\begin{equation*}
d_p(G, G')=1, \ \ p\in[1,\infty).
\end{equation*}
\end{example}
The following example is more convincing.
\begin{example}\label{example:distance approach one}
Let $G'=P_2$ be the path on two vertices and $G_N$ an unweighted (i.e. $\theta_{xy}=1$ for every edge $xy$) complete graph on $N$ vertices. Then it is known that
\begin{equation}\label{e:spectral of complete graph}\sigma(G_N)=\{0,\underbrace{\frac{N}{N-1}, \ldots, \frac{N}{N-1}}_{N-1}\}.\end{equation}
Therefore we have
\begin{equation*}
d_p(G_N,G')=\left[\left(\frac{1}{2}-\frac{1}{N}\right)\frac{N^p}{(N-1)^p}+\frac{1}{2}\left(2-\frac{N}{N-1}\right)^p\right]^{\frac{1}{p}}.
\end{equation*}
In particular, $d_1(G_N, G')=1-\frac{1}{N-1}$. Observe that
\begin{equation*}
\lim_{N\rightarrow +\infty}d_p(G_N, G')=1.
\end{equation*}
\end{example}

\section{The proof of the diameter estimate}\label{s:maintheorem}
This section is devoted to the proofs of Theorem \ref{t:diam estimate}, \ref{t:measure version} and Theorem
\ref{t:analytic version}. We first prove some lemmata.

We call a function $f:[0,1]\to [0,2]$ an \emph{admissible $2$-step
function} if there exist $a\in[0,\frac12]$ and $b\in[\frac12,1]$
such that
\begin{equation}\label{e:step admissible}f(x)=\left\{\begin{array}{ll}0, & 0\leq x< a, \\
\frac{2b-1}{b-a}, & a\leq x<b,\\
2,& b\leq x\leq 1.
\end{array}
\right.\end{equation} In particular, we say $f$ jumps at $a$ and
$b$. Clearly, $\int_0^1f(x)dx=1$ and $\int_0^2f^{-1}(x)dx=1.$ The
name for a $2$-step function is evident from the graph of the
function. In particular, any inverse function $F^{-1}_G$ of a
cumulative distribution function of a graph $G$ with $3$ vertices is
an admissible $2$-step function.

\begin{lemma}\label{l:2-step estimate}
Let $f,g$ be admissible $2$-step functions on $[0,1]$. Then we have
\begin{equation}\label{e:bounds for 2 step admissible functions}\int_0^1|f-g|(x)dx\leq 1,\end{equation}
where "$=$" holds if and only if (ignoring the order of $f,g$)
\begin{equation}\label{e:rigidity2stepfunction}
f(x)=\left\{
       \begin{array}{ll}
         0, & \hbox{$0\leq x<\frac{1}{2}$;} \\
         2, & \hbox{$\frac{1}{2}\leq x\leq 1$,}
       \end{array}
     \right.\,\,\,g(x)=1,\,\, 0\leq x\leq 1.
\end{equation}
\end{lemma}
Observe that the inverse cumulative distribution functions in Example \ref{example:distance equals one} are exactly the two functions in (\ref{e:rigidity2stepfunction}).

\begin{proof}
Let $f:[0,1]\to [0,2]$ ($g:[0,1]\to [0,2]$ resp.) be an admissible
$2$-step function jumping at $a$ and $b$ ($c$ and $d$ resp.). Denote
the height of the first jump of $f$ and $g$ by
$h_1:=\frac{2b-1}{b-a}$ and $h_2:=\frac{2d-1}{d-c}$ respectively.

The proof is divided into four cases and several subcases as follows:

\emph{{\bf Case 1.}} $0\leq a\leq c\leq \frac12\leq d\leq b\leq 2.$

\emph{Subcase 1.1.} $h_2\geq h_1.$ See Fig. \ref{F:1.1}.

For each domain $\textrm{I}$ ($\textrm{II}$ resp.) in Fig. \ref{F:1.1}, we denote by $|\textrm{I}|$
($|\textrm{II}|$ resp.) the area of that domain. We reflect the domain $\textrm{II}$
along the line $\{x=c\}$ to obtain a new domain $\textrm{II}'.$ By the fact
that $c\leq \frac{1}{2},$ we have
$$\int_0^1|f-g|=|\textrm{I}|+|\textrm{II}|=|\textrm{I}|+|\textrm{II}'|\leq \int_0^1 g= 1.$$

\emph{Subcase 1.2.} $h_2< h_1.$ See Fig. \ref{F:1.2}.

Reflect the domain $\textrm{I}$ along the line $\{x=d\}$ to obtain $\textrm{I}'.$ Then
$$\int_0^1|f-g|=|\textrm{I}|+|\textrm{II}|=|\textrm{I}'|+|\textrm{II}|\leq \int_0^2 g^{-1}(y)dy= 1.$$

\begin{figure}[h]
\begin{minipage}[t]{0.45\linewidth}
\centering
\includegraphics[scale=0.2]{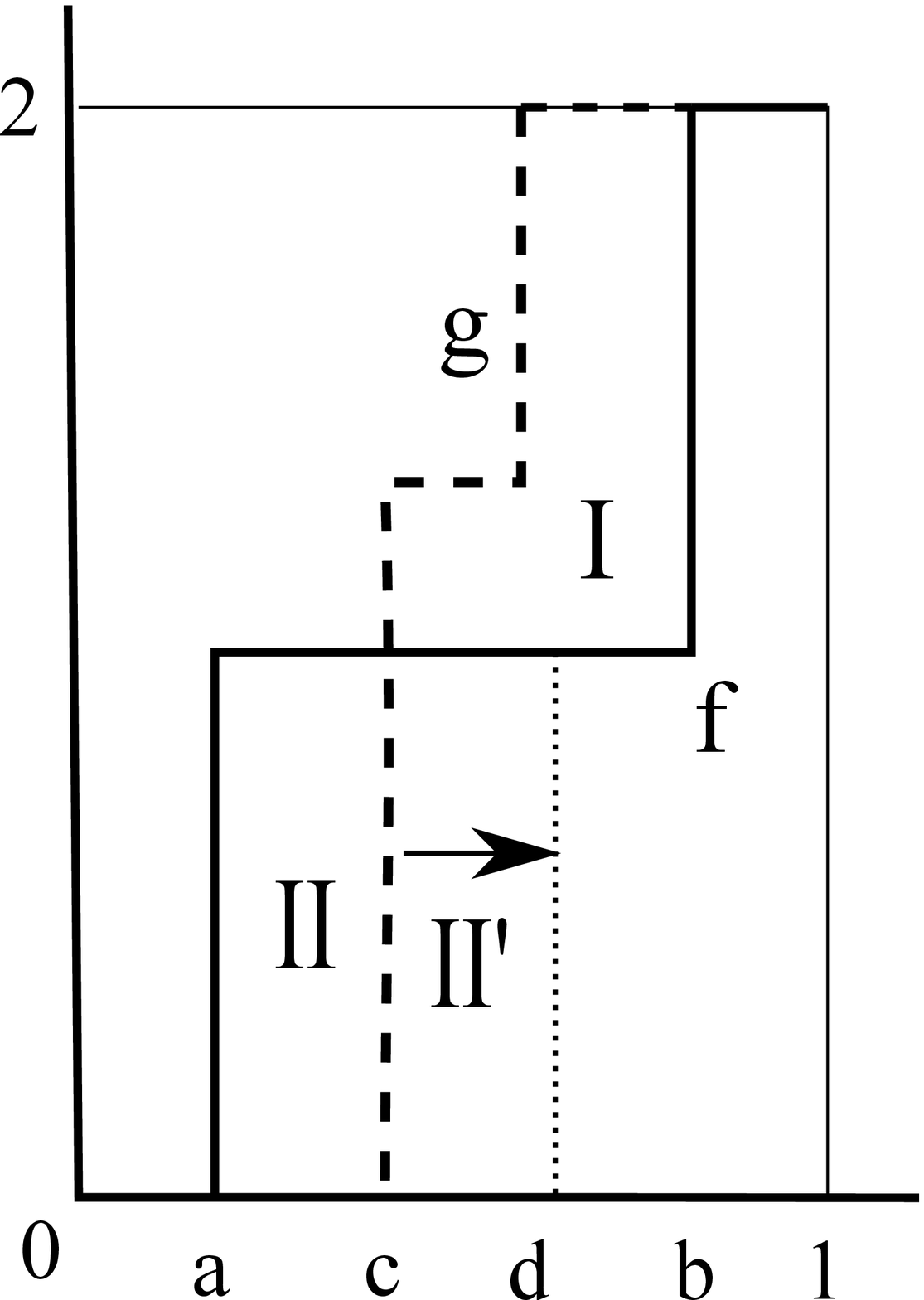}
\caption{\label{F:1.1} Subcase 1.1}
\end{minipage}
\hfill
\begin{minipage}[t]{0.45\linewidth}
\centering
\includegraphics[scale=0.2]{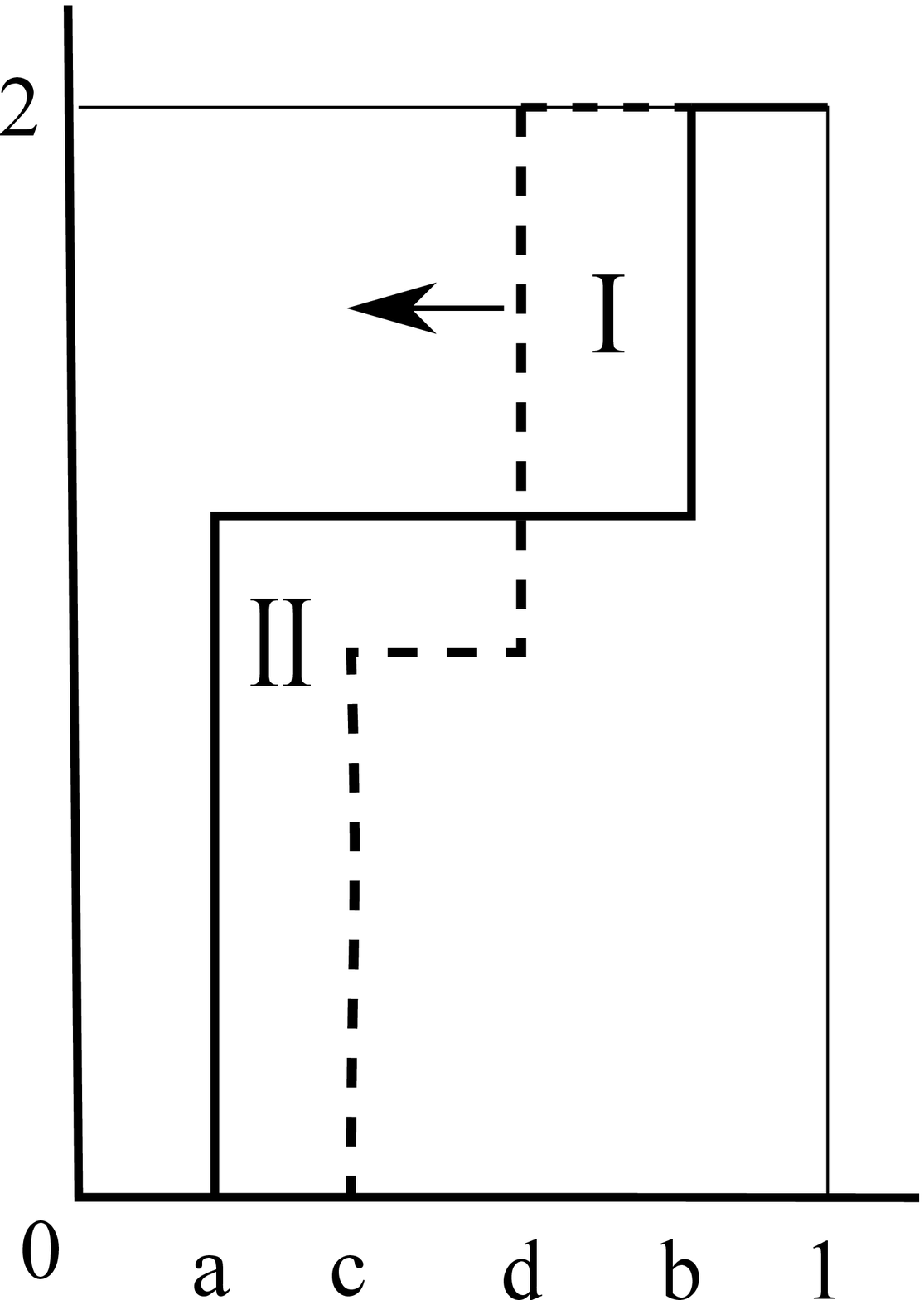}
\caption{\label{F:1.2} Subcase 1.2}
\end{minipage}
\end{figure}

%\begin{figure}[h]
%\begin{minipage}[t]{0.45\linewidth}
%\centering
%\includegraphics[width=\textwidth]{Fig1_1.eps}\label{F:1.1}
%\caption{(4,8,8) in $\mathds{R}^2$ \label{5}}
%\end{minipage}
%\hfill
%\begin{minipage}[t]{0.45\linewidth}
%\centering
%\includegraphics[width=\textwidth]{Fig1_2.eps}
%\caption{(4,8,8) in
%$\mathds{R}P^2\setminus\{\underline{o}\}$\label{4}}
%\end{minipage}
%\end{figure}

\emph{{\bf Case 2.}} $0\leq a\leq c\leq \frac12\leq b< d\leq 2.$

We claim that $h_1\leq h_2.$ Suppose not, by Fig \ref{F:2}, we have
$$1=\int_0^1 f>\int_0^1 g=1,$$ which is a contradiction. This proves the
claim.

\emph{Subcase 2.1.} $h_1\geq 1,$ see Fig. \ref{F:2.1}.

Reflect the domain $\textrm{II}$ along the line $\{y=h_1\}$ to get $\textrm{II}'.$
Since $h_1\geq 1,$
$$\int_0^1|f-g|=|\textrm{I}|+|\textrm{II}|+|\textrm{III}|=|\textrm{I}|+|\textrm{II}'|+|\textrm{III}|\leq \int_0^1 f=1.$$

\begin{figure}[h]
\begin{minipage}[t]{0.45\linewidth}
\centering
\includegraphics[scale=0.2]{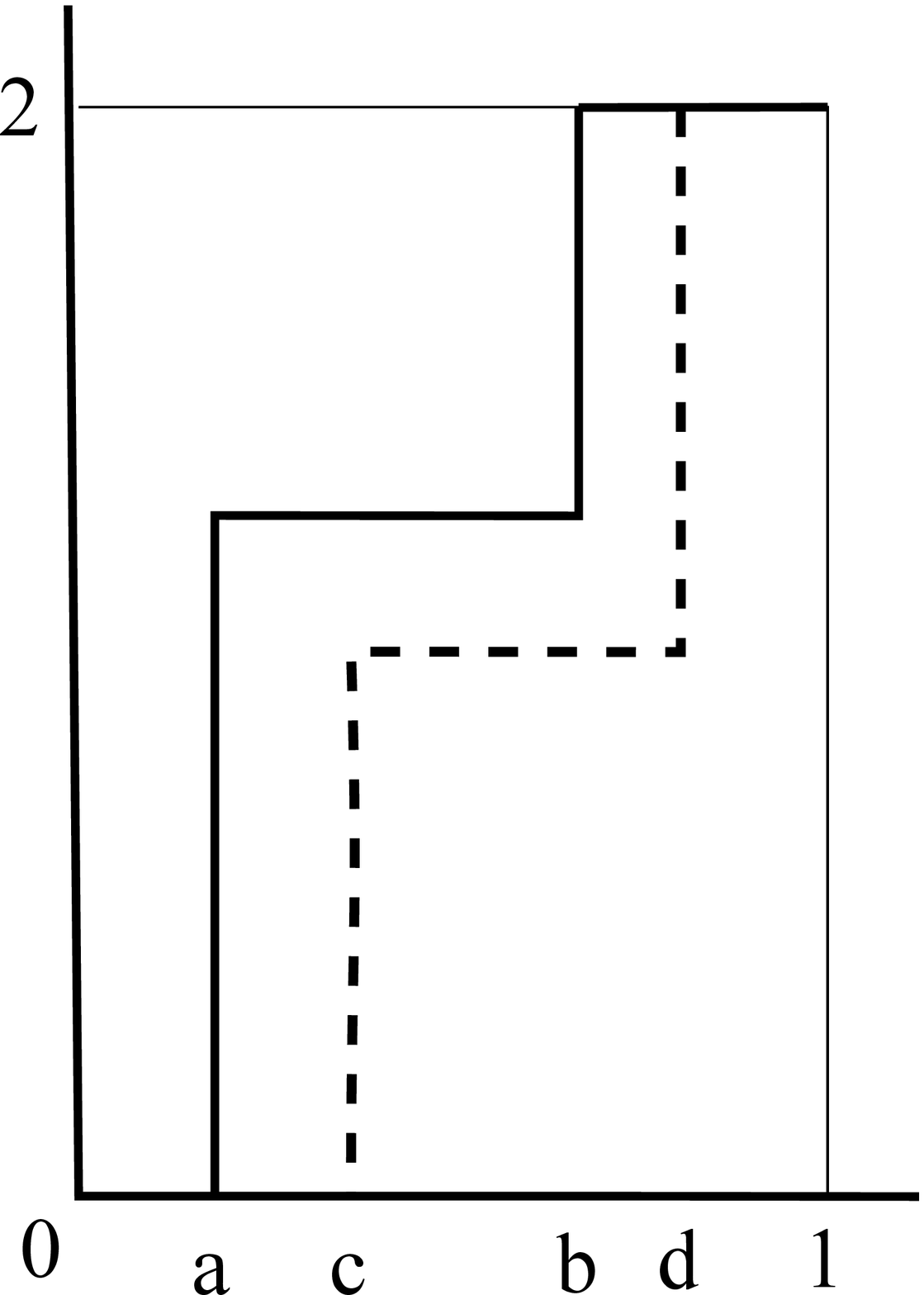}
\caption{\label{F:2} Case 2: the proof of $h_1\leq h_2$}
\end{minipage}
\hfill
\begin{minipage}[t]{0.45\linewidth}
\centering
\includegraphics[scale=0.2]{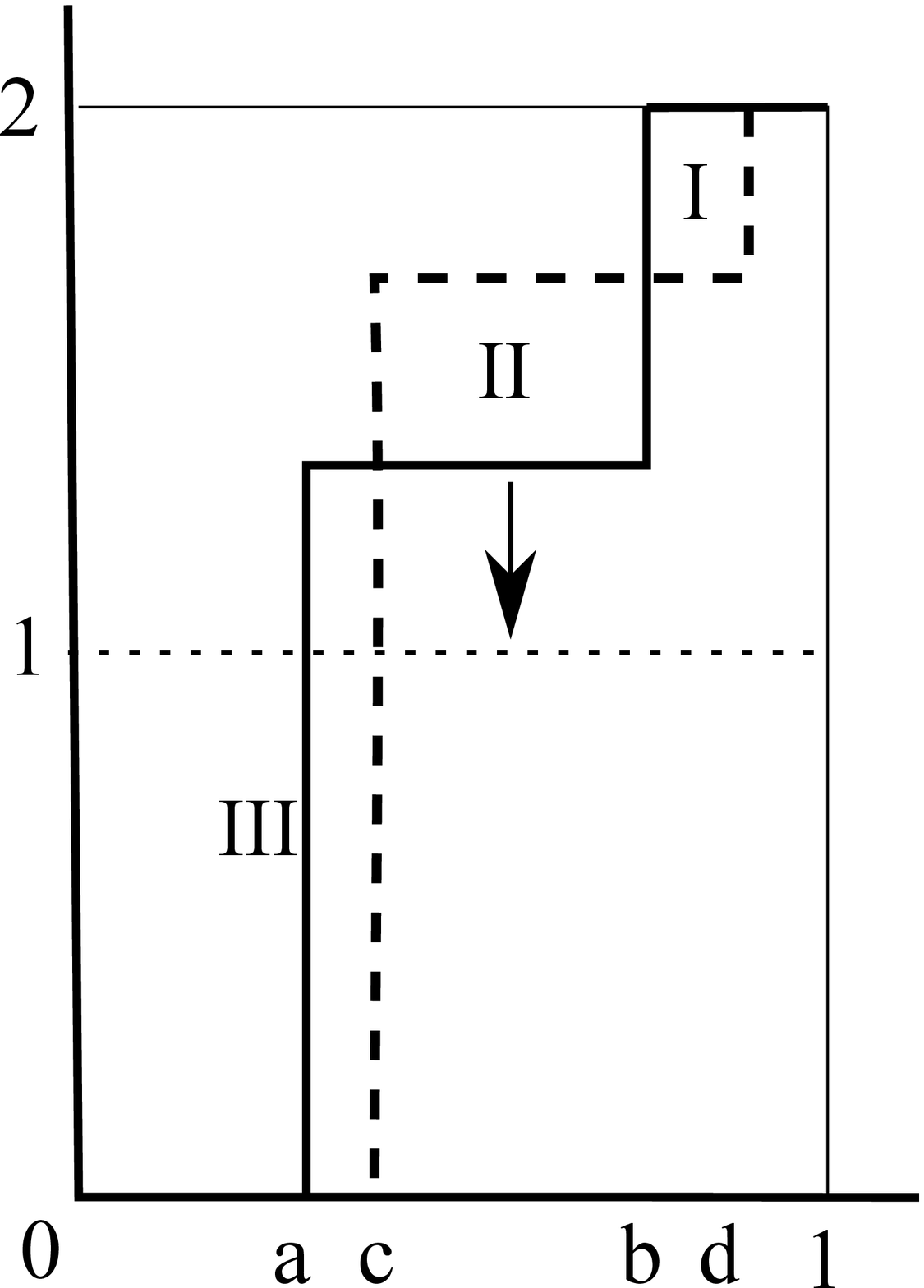}
\caption{\label{F:2.1} Subcase 2.1}
\end{minipage}
\end{figure}

\emph{Subcase 2.2.} $h_1<1.$ Further, we divide it into more
subcases.

\emph{Subcase 2.2.1.} $h_2\leq 1,$ see Fig. \ref{F:2.2.1}.

Reflect the domain $\textrm{II}$ along the line $\{y=h_2\}$ to have $\textrm{II}'.$ By
$h_2\leq 1,$
$$\int_0^1|f-g|=|\textrm{I}|+|\textrm{II}|+|\textrm{III}|=|\textrm{I}|+|\textrm{II}'|+|\textrm{III}|\leq \int_0^2g^{-1}(y)dy=1.$$

\emph{Subcase 2.2.2.} $h_2> 1.$ Moreover,

\emph{Subcase 2.2.2.1.} $h_2-h_1\leq 1.$

Then by the basic estimate,
\begin{eqnarray*}\int_0^1|f-g|&=&|\textrm{I}|+|\textrm{II}|+|\textrm{III}|=(2-h_2)(d-b)+(h_2-h_1)(b-c)+h_1(c-a)\\
&\leq& d-b+b-c+c-a=d-a \ \ \ \ ({\rm by}\ \max\{2-h_2,h_2-h_1,h_1\}\leq 1)\\
&\leq& 1.\end{eqnarray*}

\emph{Subcase 2.2.2.2.} $h_2-h_1> 1,$ see Fig. \ref{F:2.2.2.2}.

Reflect $\textrm{I}$ along the line $\{y=h_2\}$ to obtain $\textrm{I}',$ and $\textrm{III}$
along the line $\{x=c\}$ to obtain $\textrm{III}'.$ Then by the fact
$h_2-h_1\geq 1\geq 2-h_2,$ $\textrm{I}'\cap \textrm{III}'=\emptyset.$ Thus,
$$\int_0^1|f-g|=|\textrm{I}|+|\textrm{II}|+|\textrm{III}|=|\textrm{I}'|+|\textrm{II}|+|\textrm{III}'|\leq \int_0^1 g=1.$$

\begin{figure}[h]
\begin{minipage}[t]{0.45\linewidth}
\centering
\includegraphics[scale=0.2]{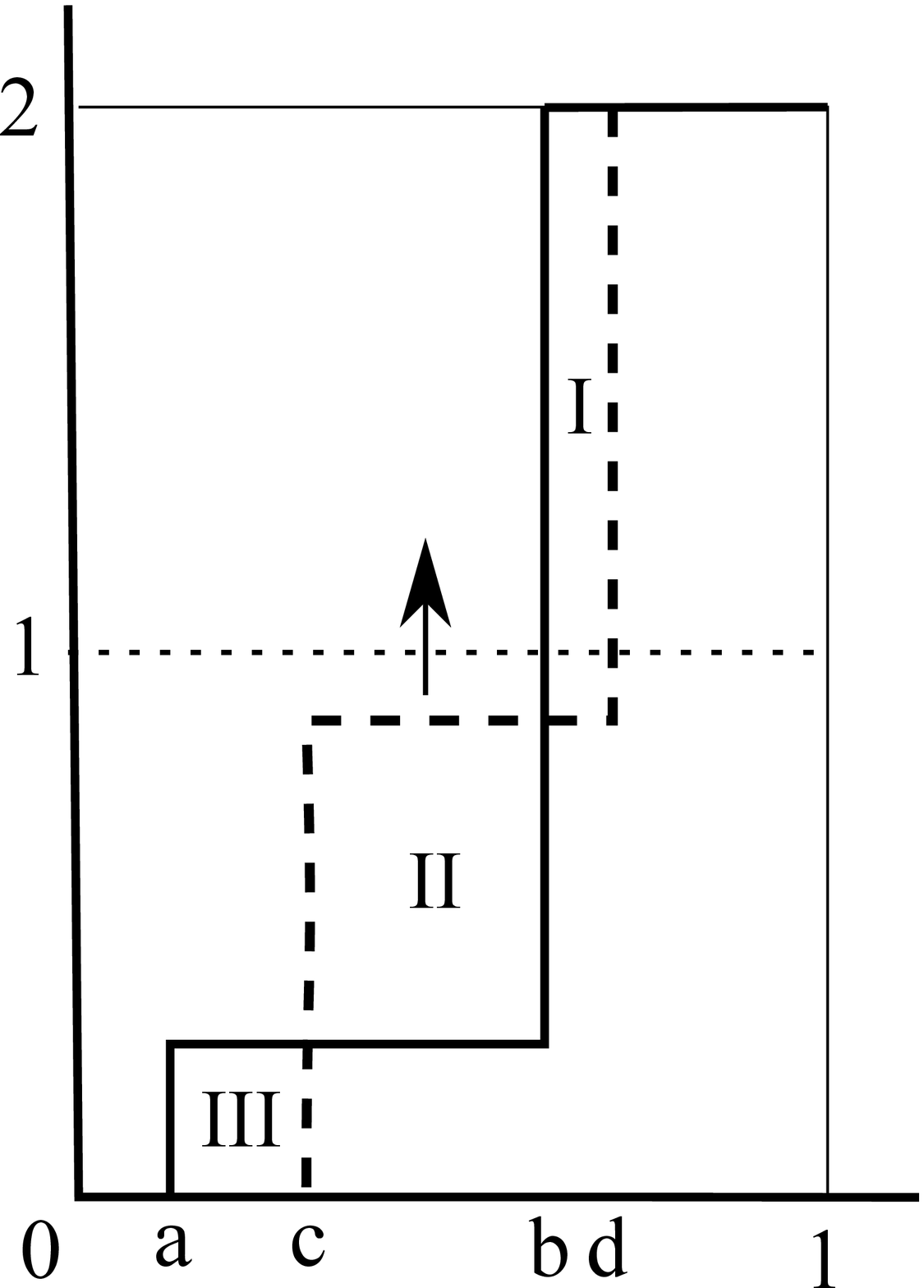}
\caption{\label{F:2.2.1} Subcase 2.2.1}
\end{minipage}
\hfill
\begin{minipage}[t]{0.45\linewidth}
\centering
\includegraphics[scale=0.2]{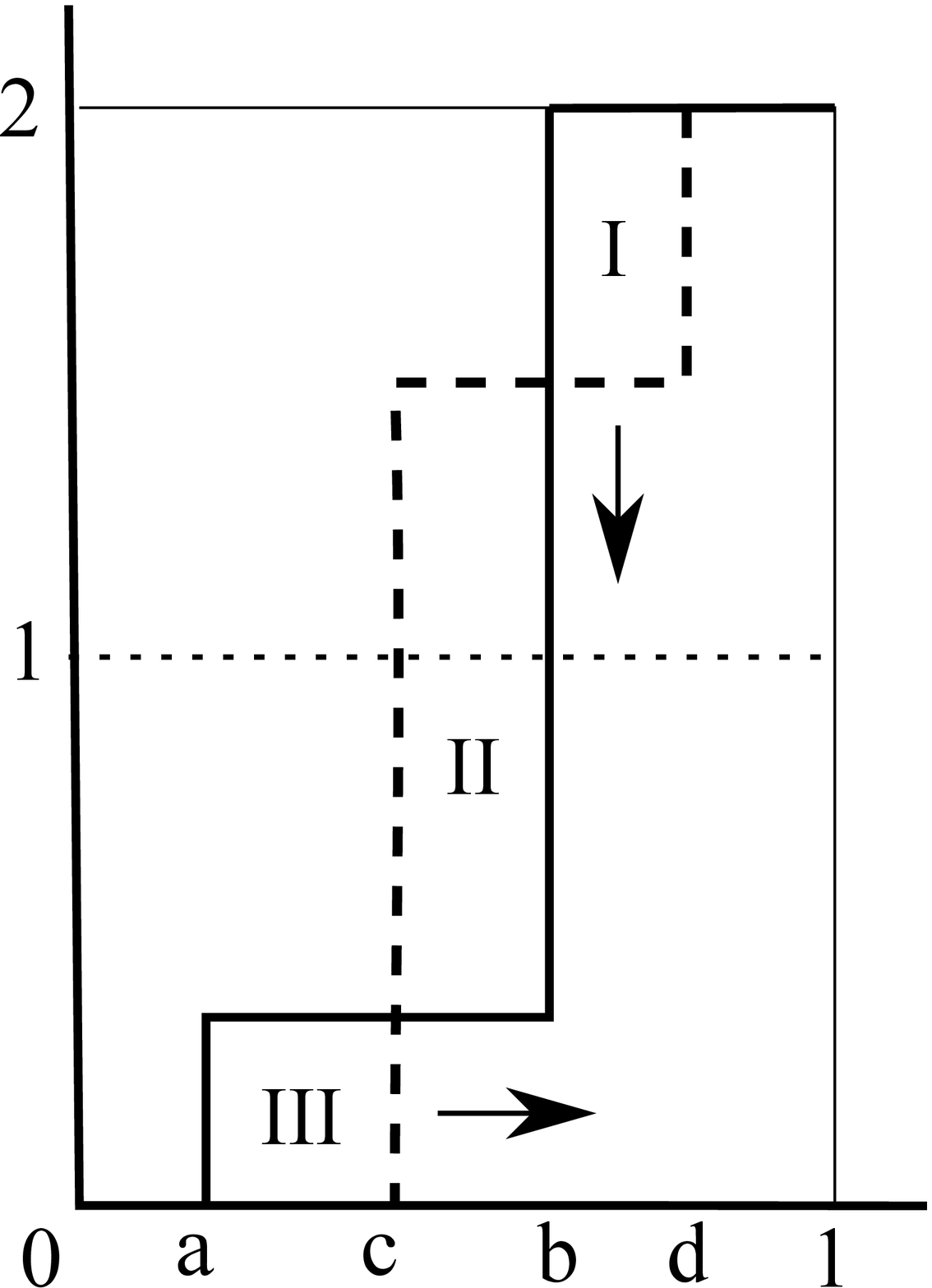}
\caption{\label{F:2.2.2.2}Subcase 2.2.2}
\end{minipage}
\end{figure}

\emph{{\bf Case 3.}} $0\leq c< a\leq \frac12\leq b< d\leq 2.$
By interchanging the role of $a,b$ and $c,d,$ this reduces to the
Case 1.

\emph{{\bf Case 4.}} $0\leq c< a\leq \frac12\leq d\leq b\leq 2.$
This reduces to Case 2 by the same change as in Case 3.

Combining all the cases and subcases, we prove (\ref{e:bounds for 2 step admissible functions}). Finally, we can check case by case that the equality in
(\ref{e:bounds for 2 step admissible functions}) can be achieved only when $f$ and $g$ are the functions given by the relation (\ref{e:rigidity2stepfunction}).
%And in Subcase 2.2.2.2, it is impossible to have equality, since we can not ensure $2-h_2, h_2-h_1, h_1=1$ simultaneously.
This completes the proof.
\end{proof}

Before proving the next lemma, we recall some basic facts from the
convex analysis. Let $\Omega$ be a convex subset of $\R^N,$ possibly
having lower Hausdorff dimension. A function $f:\Omega\to \R$ is
called \emph{convex} if for any $x,y\in \Omega$ and $0\leq t\leq 1,$
$$f(tx+(1-t)y)\leq tf(x)+(1-t)f(y).$$ In particular, for any norm $\|\cdot\|$ on $\R^N$, the function
$f:\R^N\to\R$ defined by $f(x)=\|x-x_0\|$ for some fixed $x_0$ is a
convex function. We say a point $x\in \Omega$ is \emph{extremal} if it cannot
be written as the nontrivial convex combination of two other points
in $\Omega,$ i.e. if $x=tx_1+(1-t)x_2$ for some $0<t<1$ and
$x_1,x_2\in \Omega,$ then $x=x_1=x_2.$ The set of extremal points of
a convex set $\Omega$ is denoted by $\ex(\Omega).$ A subset
$P\subset\R^N$ is called a (closed) convex polytope if it is the
intersection of finite many half spaces, i.e. there exist $K\in \N$
linear functions $\{L_j\}_{j=1}^K$ on $\R^N$ such that
$$P=\bigcap_{j=1}^K\{x\in\R^N: L_j(x)\leq 0\}.$$
We state a well-known fact which will be used to prove the next
lemma.

\begin{fact}\label{f:fact on convex}
Let $P$ be a compact convex polytope in $\R^N$ and $f:P\to\R$ a
convex function. Then
\begin{equation}
\max_{P}f=\max_{\ex(P)}f.
\end{equation}
\end{fact}

The following lemma is the special case of Theorem \ref{t:diam estimate} when two graphs have the same number of vertices.
\begin{lemma}\label{l:vector version result}
Let $N\geq 1.$ Assume that $\al=(\al_i)_{i=1}^N$ and
$\be=(\be_i)_{i=1}^N$ satisfy $0\leq \al_1\leq \cdots\leq \al_N\leq
2$ and $0\leq \be_1\leq \cdots\leq \be_N\leq 2$ and
$$\|\al\|_{\ell^1}=\|\be\|_{\ell^1}=N.$$ Then we have
$$\|\al-\be\|_{\ell^1}\leq N.$$
\end{lemma}

\begin{proof}
Let $P$ denote the compact convex polytope $\{\al\in \R^N: 0\leq
\al_1\leq \cdots\leq \al_N\leq 2, \|\al\|_{\ell^1}=N\}.$ Then by the
induction on $N,$ one can show that the set of extremal points of
$P$ is
$$\ex(P)=\left\{(\underbrace{0,\cdots,0}_k,\underbrace{a,\cdots,a}_{N-k-l},\underbrace{2,\cdots,2}_l):
0\leq k,l\leq \frac{N}{2}, a=\frac{N-2l}{N-k-l}\right\}.$$

We divide the interval $[0,1]$ equally into $N$ subintervals
$\{[\frac{i-1}{N},\frac{i}{N}]\}_{i=1}^N.$ Then for any $\al\in P,$
we define a step function $f_{\al}:[0,1]\to [0,2]$ by
$$f_{\al}|_{[\frac{i-1}{N},\frac{i}{N}]}=\al_i.$$ Clearly, $\int_0^1
f_{\al}=\frac{1}{N}\|\al\|_{\ell^1}=1.$ In addition, for any
$\gamma\in \ex(P),$ $f_{\gamma}$ is an admissible $2$-step function
defined in \eqref{e:step admissible}.

Note that for any fixed $\be_0\in \R^N,$ the function
$F:\R^N\ni\al\mapsto\|\al-\be_0\|_{\ell^1}\in \R$ is a convex
function on $\R^N.$ We claim that
\begin{equation}\label{e:maximum}\max_{\substack{\al\in
P\\\be\in P}}\|\al-\be\|_{\ell^1}=\max_{\substack{\gamma\in \ex(P)\\
\theta\in \ex(P)}}\|\gamma-\theta\|_{\ell^1}.\end{equation} By Fact
\ref{f:fact on convex},
\begin{eqnarray*}
\max_{\substack{\al\in P\\\be\in P}}\|\al-\be\|_{\ell^1}&=&\max_{\be\in P}\max_{\al\in P}\|\al-\be\|_{\ell^1}=\max_{\be\in P}\max_{\gamma\in \ex(P)}\|\gamma-\be\|_{\ell^1}\\
&=&\max_{\gamma\in \ex(P)}\max_{\be\in
P}\|\gamma-\be\|_{\ell^1}=\max_{\gamma\in \ex(P)}\max_{\theta\in
\ex{(P)}}\|\gamma-\theta\|_{\ell^1}.
\end{eqnarray*} This proves the claim.

For any $\gamma,\theta\in \ex(P),$ noting that $f_{\gamma}$ and
$f_{\theta}$ are admissible $2$-step functions, by Lemma
\ref{l:2-step estimate}, we have
$$\|\gamma-\theta\|_{\ell^1}=N\int_0^1|f_{\gamma}-f_{\theta}|\leq N.$$

Combining this with \eqref{e:maximum}, we prove the lemma.
\end{proof}

Now we can prove Theorem \ref{t:analytic version}. A function
$f:[0,1]\to [0,2]$ is called \emph{a rationally distributed step function}
if there is a (rational) partition $0=r_0< r_1<r_2<\cdots<r_N=1$
with $r_i\in \Q$ for all
$0\leq i\leq N$ and an increasing sequence $0\leq a_1< \cdots< a_N\leq 2$ such that \[f(x)=\left\{\begin{array}{ll}a_1, & 0\leq x< r_1, \\
a_2, & r_1\leq x<r_2,\\
\,\,\,\,\vdots\\
a_N,& r_{N-1}\leq x\leq 1.
\end{array}
\right.\]

\begin{proof}[Proof of Theorem \ref{t:analytic version}]
First, we consider $p=1.$
By the standard approximation argument, any such functions, $f$ and
$g$, can be approximated in $L^1$ norm by a sequence of rationally
distributed step functions, say $\{f_n\}_{n=1}^{\infty}$ and
$\{g_n\}_{n=1}^{\infty},$ satisfying $\int_{0}^1f_n=\int_0^1g_n=1.$
Hence it suffices to prove the theorem for rationally distributed
step functions.

W.l.o.g., we may assume $f$ and $g$ are rationally distributed step
functions, say $f|_{[r_{i-1},r_i]}=a_i$ for $1\leq i\leq L$ and
$g|_{[t_{j-1},t_j]}=b_j$ for $1\leq j\leq K$ where $L,K\in \N.$ Let
$N$ denote the least common multiple of
$\{m_i\}_{i=1}^L\cup\{n_j\}_{j=1}^K$ where $m_i,n_j$ are the
denominators of $r_i=\frac{c_i}{m_i}$ and $t_j=\frac{d_j}{n_j}$
($c_i,m_i,d_j,n_j\in\N$), $1\leq i\leq L,1\leq j\leq K.$ Then we
have for any $1\leq p\leq N$
$$f|_{[\frac{p-1}{N},\frac{p}{N}]}=\al_p,$$
$$g|_{[\frac{p-1}{N},\frac{p}{N}]}=\be_p,$$ where
$\al_p=a_l$ and $\be_p=b_k$ for some $1\leq l\leq L,1\leq k\leq K.$
Obviously, $0\leq \al_1\leq \cdots\leq \al_{N}\leq 2,$ $0\leq
\be_1\leq \cdots\leq \be_{N}\leq 2$ and
$$\|\al\|_{\ell^1}=\|\be\|_{\ell^1}=N.$$ Hence Lemma \ref{l:vector version
result} implies that
$$\|\al-\be\|_{\ell^1}\leq N.$$ That is,
$$\int_{0}^1|f-g|\leq 1.$$

For $p\in (1,\infty),$ it can be easily derived from the result for $p=1.$
\begin{eqnarray*}\int_{0}^{1}|f-g|^p&\leq& 2^{p-1}\int_{0}^1|f-g|\\
&\leq&2^{p-1}.
\end{eqnarray*}

This proves the theorem.\end{proof}

Theorem \ref{t:measure version} then follows directly.
\begin{proof}[Proof of Theorem \ref{t:measure version}]
Let $F_{\mu}$ and $F_{\nu}$ denote the cumulative distribution functions of the
measures $\mu$ and $\nu$ respectively. Since the total area of the
square $[0,1]\times[0,2]$ is equal to $2$, by the assumption
$m_1(\mu)=m_1(\nu)=1$ we have
$$\int_{0}^1F_{\mu}^{-1}(x)dx=\int_{0}^1F_{\nu}^{-1}(x)dx=1.$$ Then
our theorem follows from Theorem \ref{t:analytic version} and Lemma \ref{l:cumulative dis}.
\end{proof}

Now we can prove Theorem \ref{t:diam estimate}.
\begin{proof}[Proof of Theorem \ref{t:diam estimate}]
This follows from Theorem \ref{t:measure version} directly.
\end{proof}

%Let $G_1$ and $G_2$ be two graphs with $N_1$ and $N_2$ vertices,
%respectively. Denoted by $F_{G_1}^{-1}$ and $F_{G_2}^{-1}$ the
%inverse functions of the cumulative distribution functions $F_{G_1}$
%and $F_{G_2}.$ In particular, $F_{G_1}^{-1}$ ($F_{G_2}^{-1}$ resp.)
%are step functions on $[0,1]$ distributed on the partition
%$\{[\frac{i-1}{N_1},\frac{i}{N_1}]\}_{i=1}^{N_1}$
%($\{[\frac{i-1}{N_2},\frac{i}{N_2}]\}_{i=1}^{N_2}$ resp.). Consider
%a refined partition of these two,
%$\{[\frac{j-1}{N_1N_2},\frac{j}{N_1N_2}]\}_{j=1}^{N_1N_2}$. Then
%both of $F_{G_1}^{-1}$ and $F_{G_2}^{-1}$ are step functions on this
%refined partition, which corresponds to two vectors $\al$ and $\be$
%in $\R^{N_1N_2}$ by setting $\al_j=F^{-1}_{G_1}(x)$ and
%\be_j=F^{-1}_{G_2}(x)$ for some $x\in
%(\frac{j-1}{N_1N_2},\frac{j}{N_1N_2}).$ Obviously, $0\leq \al_1\leq
%\cdots\leq \al_{N_1N_2}\leq 2,$ $0\leq \be_1\leq \cdots\leq
%\be_{N_1N_2}\leq 2$ and
%$$\|\al\|_{\ell^1}=\|\be\|_{\ell^1}=N_1N_2.$$ Hence Lemma \ref{l:vector version
%result} implies that
%$$\|\al-\be\|_{\ell^1}\leq N_1N_2.$$ That is,
%$$\int_0^1|F^{-1}_{G_1}-F^{-1}_{G_2}|\leq 1.$$ This proves the
%theorem.

\section{Spectral distances of infinite graphs}\label{s:spectral distances of infinite graphs}
In this section, we introduce two definitions of spectral measures
for infinite weighted graphs with countable vertex set and extend our approach of spectral distance to this setting.

\subsection{Spectral measures by exhaustion}
Let $G=(V,E,\theta)$ be an infinite weighted graph
and $G_{\Omega}:=(\Omega, E_{|\Omega}, \theta_{|\Omega\times\Omega})$ a finite connected subgraph of $G$ induced by a subset $\Omega\subset V$. We
introduce the Dirichlet boundary problem of the normalized Laplacian
on $\Omega,$ see e.g. \cite{BHJ12}. Let $\ell^2(\Omega,\theta)$
denote the space of real-valued functions on $\Omega$. Note that
every function $f\in\ell^2(\Omega,\theta)$ can be extended to a
function $\tilde{f}\in\ell^2(V,\theta)$ by setting $\tilde{f}(x)=0$
for all $x\in V\setminus\Omega$. The normalized Laplacian with the
Dirichlet boundary condition on $\Omega,$ denoted by
$\Delta_{G_\Omega},$ is defined as $\Delta_{G_\Omega}: \ell^2(\Omega,\theta)
\to \ell^2(\Omega,\theta)$,
$$\Delta_{G_\Omega} f = (\Delta_G \tilde{f})_{|\Omega}.$$ Thus for $x\in
\Omega$ the Dirichlet normalized Laplacian is pointwise defined by
$$\Delta_{G_\Omega} f(x) = f(x) - \frac{1}{\theta_x}\sum_{y\in\Omega} \theta_{xy}f(y) =
\tilde{f}(x) - \frac{1}{\theta_x}\sum_{y\in V}
\theta_{xy}\tilde{f}(y),$$
where $\theta(x)$ is the weighted degree of the entire graph. A simple calculation shows that
$\Delta_{G_\Omega}$ is a positive self-adjoint operator. We arrange the
eigenvalues of the Dirichlet Laplace operator $\Delta_{G_\Omega}$ in
nondecreasing order, i.e. $\lambda_1(\Omega)\leq \lambda_2(\Omega) \leq
\ldots \leq \lambda_N(\Omega) $, where $N$ is the cardinality of the
set $\Omega$, i.e. $N = |\Omega|$. By the trace condition, we
also have the key property $$\sum_{i=1}^N\lambda_i(\Omega)=N.$$ As same as
finite graphs, we associate it with the spectral measure,
$$\mu_{\Omega}=\frac{1}{N}\sum_{i=1}^N\delta_{\lambda_i(\Omega)}.$$ Hence $m_1(\mu_{\Omega})=1.$

A sequence of finite connected subgraphs
$\{\Omega_n\}_{n=1}^{\infty}$ is called an \emph{exhaustion} of the
infinite graph $G$ if $\Omega_n\subset\Omega_{n+1}$ for all $n\in
\N$ and $\cup_{n=1}^{\infty}\Omega_n=V.$ Hence we have a sequence of
probability measures $\{\mu_{\Omega_n}\}_{n=1}^{\infty}$ on $[0,2].$
Since $P([0,2])$ is compact under the weak topology, up to a
subsequence, w.l.o.g. we have $\mu_{\Omega_n}\rightharpoonup \mu$
for some $\mu\in P([0,2]).$ Note that any subsequence of an exhaustion is still an exhaustion. Therefore we define the spectral measures of
an infinite graph by all possible exhaustions. Note that the
convergence of the spectral structure was studied in more general
setting by Kuwae-Shioya \cite{Kuwae2003}.
\begin{definition}\label{d:infinite by exhaustion}
Let $G$ be an infinite weighted graph. We define the
\emph{spectral measures} of $G$ by exhaustion as
$$\SM(G):=\{\mu\in P([0,2]): \mbox{there\ is\ an\ exhaustion}\ \{\Omega_n\}_{n=1}^{\infty}\ s.t.\ \mu_{\Omega_n}\rightharpoonup\mu \}.$$
\end{definition}

One can show that $\SM(G)$ is a closed subset of $P([0,2])$. Since
$m_1(\mu_{\Omega_n})=1$ for any $n\in \N,$ by the weak convergence, we
have $m_1(\mu)=1$ for any $\mu\in \SM(G).$

For any metric space $(X,d),$ one can define the \emph{Hausdorff distance}
between the subsets of $X$. For any subset $A\subset X,$ we define
the distance function to the subset $A$ as $X\ni x\mapsto
d(x,A)=\inf\{d(x,y)|y\in A\},$ and the $r$-neighborhood of $A$ as
$U_r(A):=\{y\in X| d(y,A)<r\},$ $r>0.$ Given two subsets $A,B\subset
X,$ the Hausdorff distance between them is defined as
$$d_{H}(A,B):=\inf\{r>0| A\subset U_r(B), B\subset U_r(A)\}.$$
%Simple calculation shows that $d_H(A,B)=\max\{\sup_{x\in
%A}d(x,B),\sup_{y\in B}d(y,A)\}.$
One can show that the set of closed subsets of $X$ endowed with the
Hausdorff distance is a metric space.

Note that for $p\in[1,\infty),$ $P([0,2])$ endowed with the $p$-th
Wasserstein distance is a metric space and $\SM(G)$ is a closed
subset of $P([0,2])$ for any weighted graph $G.$ We denote by $\G$
the collection of all (possibly infinite) weighted graphs. Hence
$\G$ endowed with the Hausdorff distance induced from
$(P([0,2]),d_p^W)$, denoted by $d_{p,H},$ is a pseudo-metric space.

A direct application of Theorem \ref{t:measure version} yields
\begin{thm}\label{T:diameterInfi1} For $p\in [1,\infty),$
$$\diam(\G,d_{p,H})\leq 2^{1-\frac{1}{p}}.$$
\end{thm}

\subsection{Spectral measures for random rooted graphs}\label{s:random rooted graphs}
We follow Benjamini-Schramm \cite{Benjamini2001a}, Aldous-Lyons \cite{Aldous2007} and Ab\'ert-Thom-Vir\'ag \cite{ATV}
to define random rooted graphs.

For any $D\geq 1,$ we define a subcollection of $\G,$
$\G_D:=\{(V,E,\theta)\in \G| \deg_x\leq D, \theta_{xy}\leq D \mbox{\
for \ all\ }x,y\in V\}$ where $\deg_x=|\{y\in V| y\sim x\}|$, i.e. the set of weighted graphs with
bounded (unweighted) degree ($\leq D$) and bounded edge weights
($\leq D$). Let $\RG_D$ denote the set of graphs $G$ in $\G_D$ with
a distinguished vertex, called the root of $G.$

For any $x,y\in V$ of $G=(V,E,\theta),$ we denote by $d_C(x,y)$ the distance between $x$ and $y,$ i.e. $d_C(x,y):=\inf\{n|
\mbox{\ there\ exist}\ \{x_i\}_{i=0}^{n}\ s.t.\ x= x_0\sim
x_1\sim\cdots\sim x_{n}= y\},$ and by $B_k(x):=\{z\in V|
d_C(x,z)\leq k\},$ $k\in \N\cup\{0\}$, the ball of radius $k$
centered at $x.$ Let $(G_1,o_1)$ and $(G_2,o_2)$ be two rooted
graphs with distinguished vertices $o_1$ and $o_2$, respectively. We call that $B_k(o_1)$ is isomorphic to $B_k(o_2)$ if there
exists a bijective map $f:B_k(o_1)\to B_k(o_2)$ such that
$f(o_1)=f(o_2)$ and $x\sim y$ for $x,y\in B_k(o_1)$ if and only if
$f(x)\sim f(y).$ For $(G_1,o_1),(G_2,o_2)\in \RG_D$ with
$G_1=(V_1,E_1,\theta_1)$ and $G_2=(V_2,E_2,\theta_2),$ we define the
\emph{rooted distance} between $G_1$ and $G_2$ as $1/K$ where
\begin{align*}
K=\max\{k\in \N|\ \exists \mbox{ an\ isomorphism\ }f:B_k(o_1)\to B_k(o_2)\\
 \text{ such that }\ \sup_{x,y\in B_k(o_1)}|\theta_{1,xy}-\theta_{2,f(x)f(y)}|\leq \frac1k\},
\end{align*}
$\theta_{1,xy},\theta_{2,f(x)f(y)}$ are edge weights of $xy\in E_1$, $f(x)f(y)\in E_2$, respectively.
One can prove that
 $\RG_D$ endowed with the rooted distance is a compact metric space.

By a \emph{random rooted graph} of degree $D$ we mean a Borel
probability distribution on $\RG_D.$ We denote by $\RRG_D$ the
collection of random rooted graphs of degree $D.$ Any finite
weighted graph $G$ gives rise to a random rooted graph by assigning the
root of $G$ uniformly randomly.

For a rooted weighted graph $(G,o)\in \RG_D$ with $G=(V,E,\theta),$
the normalized Laplacian is a bounded self-adjoint operator on
$\ell^2(V,\theta)$ which is independent of $o.$ By spectral theorem,
there is a projection-valued measure, denoted by $P_{\bullet},$ on
$[0,2],$ i.e. $P_A$ is a projection on $\ell^2(V,\theta)$ for any
Borel $A\subset [0,2],$ such that for any continuous function $f\in
C([0,2])$ we have the functional calculus
\begin{equation}\label{e:functional calculus}f(\Delta_{G})=\int_{[0,2]}f(x)dP_x\end{equation} where $P_x=P_{[0,x]}.$ We
define the \emph{spectral measure} of the rooted graph $(G,o)$ as
$$\mu_{G,o}=\frac{1}{\theta_o}\langle P_A\delta_o,\delta_o\rangle,\ \ \forall A\subset
[0,2],$$ where $\langle\cdot,\cdot\rangle$ is the inner product for
$\ell^2(V,\theta).$ One can easily show that $\mu_{G,o}$ is a
probability measure on $[0,2]$. Further calculation by using
\eqref{e:functional calculus} yields $m_1(\mu_{G,o})=1.$ Now we can
define the expected spectral measure for rooted random graphs.
\begin{definition}
Let $G$ be a random rooted graph. We define the \emph{expected
spectral measure} of $G$ as $$\mu_G=E(\mu_{G,o})$$ where the
expectation is taken over the distribution on $\RG_D.$
\end{definition}
Let $G$ be a random rooted graph rising from a finite weighted graph
with uniform distribution on its vertices. A similar calculation as
in Ab\'ert-Thom-Vir\'ag \cite{ATV} shows that
$$\mu_G=\frac{1}{N}\sum_{i=1}^N\delta_{\lambda_i}$$ where
$\{\lambda_i\}_{i=1}^N$ is the spectrum of the finite graph. Hence
the expected spectral measure of random rooted graphs generalizes
the spectral measure of finite graphs. There are other interesting
classes of random rooted graphs such as unimodular and sofic ones, see
e.g. \cite{ATV}.

The set of random rooted graphs of degree $D,$ $\RRG_D$, endowed
with $d_p^W$ Wasserstein distance for expected spectral measures
($d_p$ in short) is a pseudo-metric space. By Theorem \ref{t:measure
version}, one can prove the following theorem.
\begin{thm}\label{T:diameterInfi2}
For $p\in [1,\infty),$
$$\diam (\RRG_D,d_p)\leq 2^{1-\frac{1}{p}}.$$
\end{thm}

\section{Calculation of examples}\label{s:examples}
From now on, we will concentrate on the study of the spectral distance $d_1$.
We calculate this distance for several classes
of graphs in this section. Rather than the exact value of the $d_1$ distance between two graphs, we are more concerned with the asymptotical behavior of the distance between two sequences of graphs which become larger and larger, as the sizes of real networks in practice nowadays are typically huge. All example graphs we consider in this section are unweighted.

\begin{pro}
For two complete graphs $G$ and $G'$ with $N$ and $M$ ($ M>N $) vertices
respectively, we have $$d_1(G, G')=2\frac{M-N}{N(M-1)}.$$
\label{T:complete}
\end{pro}
\begin{proof}
Recall the spectrum (\ref{e:spectral of complete graph}) of a complete graph.
We then calculate
the distance (i.e. the area of the grey region shown in Fig. \ref{F:4.1}),
$$d_{1}(G,G')=\frac{M}{M-1}\left(\frac{1}{N}-\frac{1}{M}\right)+\left(\frac{N}{N-1}-\frac{M}{M-1}\right)\left(1-\frac{1}{N}\right)=2\frac{M-N}{N(M-1)}.$$

\begin{figure}[h]
\centering
\includegraphics[scale=0.5]{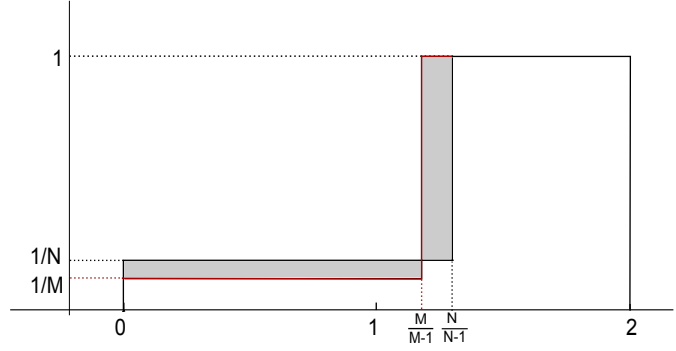}
\caption{\label{F:4.1} Two complete graphs of size $N$ and $M$}
\end{figure}
\end{proof}

\begin{remark}
When the size difference $M-N$ of two complete graphs is a fixed constant $C$, we have
%For two complete graphs with $N$ and $N+C$ ($C$ is a fixed
%constant) vertices respectively, we have
$$d_1(G,
G')=\mathcal{O}(1/N^2)\text{ as }N\rightarrow\infty.$$
\label{C:complete}
\end{remark}

%\begin{proof}
%It follows directly from Theorem \ref{T:complete} with $M=N+C$,
%$d_{1}(G,G')=2\frac{N+C-N}{N(N+C-1)}=\mathcal{O}(1/N^2).$
%\end{proof}

\begin{pro}
For two connected complete bipartite graphs
$G$ and $G'$ of size $N$ and $M$ ($ M>N $) respectively,
we have
$$d_1(G, G')=2\frac{M-N}{NM}.$$
\label{T:bi_complete}
\end{pro}
\begin{proof}
The spectrum of a complete bipartite graph $G$ with $N$ vertices
is
\begin{equation*}
 \sigma(G)=\{0, \underbrace{1, \ldots, 1}_{N-2}, 2\}.
\end{equation*}
Then the distance is
(the area of the grey region shown in Fig. \ref{F:4.2})
$$d_{1}(G,G')=\left(\frac{1}{N}-\frac{1}{M}\right)+\left(\frac{M-1}{M}-\frac{N-1}{N-1}\right)=2\frac{M-N}{NM}.$$
\begin{figure}[h]
\centering
\includegraphics[scale=0.50]{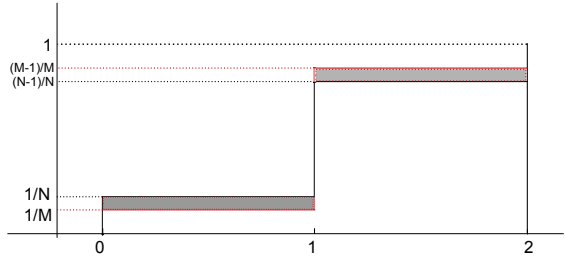}
\caption{\label{F:4.2} Two complete bipartite graphs of size $N$ and $M$}
\end{figure}
\end{proof}

\begin{remark}
If the size difference $M-N$ of
two complete bipartite graphs is a fixed constant $C$, we again observe the behavior
%
%$G$ and $G'$ with $N$ and
%$N+C$ ($C$ is a fixed constant) vertices respectively, we have
$$d_1(G,
G')=\mathcal{O}(1/N^2)\text{ as }N\rightarrow\infty.$$
\label{C:bi_complete}
\end{remark}
%
%\begin{proof}
%This follows directly from Theorem \ref{T:bi_complete} with $M=N+C$,
%$d_{1}(G,G')=2\frac{N+C-N}{N(N+C)}=\mathcal{O}(1/N^2).$
%\end{proof}

\begin{pro}
For two cubes $G$ and $G'$ of size $2^N$ and $2^{N+1}$ respectively, we have
$$d_1(G, G')=\frac{1}{N+1}.$$ \label{T:path}
\end{pro}

\begin{proof}
The spectrum of the cube $G$ with $2^N$ vertices is
$$\left\{\frac{2i}{N}\text{ with multiplicity } \binom{N}{i}, i=0,\ldots,N\right\}.$$
%Let $G'$ be the
%cube of size $2^{N+1}$.

Firstly, observe $\frac{2i}{N}=\frac{2j}{N+1}$ when $i=j=0$ or $i=N$, $j=N+1$. And for $j=i$, we
have
$$\frac{2(i-1)}{N}<\frac{2j}{N+1}<\frac{2i}{N},  \text{ for }
1\leq i\le N.$$

Secondly, by the recursive formula $\binom{N+1}{k}=\binom{N}{k-1}+\binom{N}{k}$, for $1\leq k\leq N$, we derive
$$\frac{1}{2^{N+1}}\sum\limits_{i=0}^{k}\binom{N+1}{i}<\frac{1}{2^N}\sum\limits_{i=0}^{k}\binom{N}{i}<\frac{1}{2^{N+1}}\sum\limits_{i=0}^{k+1}\binom{N+1}{i}, \text{ for } 0\leq k\leq N-1.$$
Therefore the distance between $G$ and
$G'$ equals the area of the grey region depicted in Fig. \ref{F:4.3}.
%Two density curves reach the larger values alternately, and the sum
%of cross square areas is their spectral distance.
\begin{figure}[h]
\centering
\includegraphics[scale=0.5]{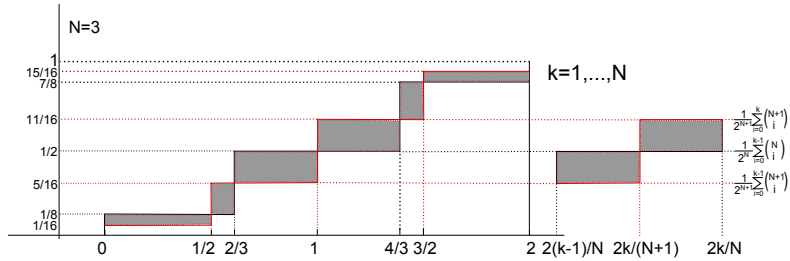}
\caption{An example of two neighboring cubes with $N=3$ and $N+1=4$.}
\label{F:4.3}
\end{figure}
Again by the recursive formula of binomial numbers, we calculate,
\begin{eqnarray*}
d_{1}(G,G')&=&\sum\limits_{k=1}^{N}\Bigg\{\left(\frac{2k}{N+1}-\frac{2(k-1)}{N}\right)\left[\frac{1}{2^N}\sum\limits_{i=0}^{k-1}\binom{N}{i}-\frac{1}{2^{N+1}}\sum\limits_{i=0}^{k-1}\binom{N+1}{i}\right]\\
& &+\left(\frac{2k}{N}-\frac{2k}{N+1}\right)\left[\frac{1}{2^{N+1}}\sum\limits_{i=0}^{k}\binom{N+1}{i}-\frac{1}{2^N}\sum\limits_{i=0}^{k-1}\binom{N}{i}\right]\Bigg\}\\  &=&\frac{1}{2^N N(N+1)}\Bigg\{\sum\limits_{k=1}^{N} (N-k+1)\left[2\sum\limits_{i=0}^{k-1}\binom{N}{i}-\sum\limits_{i=0}^{k-1}\binom{N+1}{i}\right]\\
& & +k\left[\sum\limits_{i=0}^{k}\binom{N+1}{i}-2\sum\limits_{i=0}^{k-1}\binom{N}{i}\right]\Bigg\}\\
%&=&\frac{1}{2^N N(N+1)}\sum\limits_{k=1}^{N} \Bigg\{ (N-k+1)\left[\sum\limits_{i=0}^{k-1}\binom{N}{i}-\sum\limits_{i=1}^{k-1}\binom{N}{i-1}\right]\\
%& &+k\left[\sum\limits_{i=0}^{k}\binom{N}{i}+\sum\limits_{i=1}^{k}\binom{N}{i-1}-2\sum\limits_{i=0}^{k-1}\binom{N}{i}\right]\Bigg\}\\
&=&\frac{1}{2^N N(N+1)}\sum\limits_{k=1}^{N} \left[(N-k+1)\binom{N}{k-1}+k\binom{N}{k}\right]\\
&=&\frac{2}{2^N N(N+1)}\sum\limits_{k=1}^{N}k\binom{N}{k}=\frac{2}{2^N N (N+1)}\cdot N \cdot 2^{N-1}=\frac{1}{N+1}.
\end{eqnarray*}
\end{proof}

\begin{remark}
The distance between two neighboring cubes ($N$-cube and $(N+1)$-cube)
is $\mathcal{O}(1/N)$ as $N$ tends to infinity. Recall a crucial difference of this example from previous ones is that the size difference, $2^N$, is not uniformly bounded as $N\rightarrow\infty$.
%
%does not converge to zero with $\mathcal{O}(1/N^2)$, since the
%difference in size between these two cubes is $2^{N}$, which also
%goes to infinity with $N$ goes to infinity.
\end{remark}

\begin{pro}\label{p:paths}
For two paths $G$ and $G'$
of size $N$ and $N+1$ respectively, we have
$$d_1(G, G')=\frac{1}{N(N+1)}\left(\cot^2{\frac{\pi}{2N}}-\cot^2{\frac{\pi}{2(N-1)}}+1\right)$$
\end{pro}
\begin{proof}
%The proof of this theorem is similar as Them. \ref{T:path}.
The spectrum of the path $G$ with with $N$ vertices is
$$\left\{1-\cos{\frac{\pi i}{N-1}}, i=0,1,\ldots,N-1\right\}.$$
Since $1-\cos\frac{i\pi}{N-1}<1-\cos\frac{(i+1)\pi}{N}<1-\cos\frac{(i+1)\pi}{N-1}$ for $i=0,\ldots,N-2$, and every eigenvalue of a path has multiplicity one,
the situation is similar to Proposition \ref{T:path}, as shown
%the spectral distance between $G$ and $G'$ is the sum of cross area
%of two curves reaching larger values alternately, which is shown as
%the gray area
in Fig. \ref{F:4.4}.
\begin{figure}[h]
\centering
\includegraphics[scale=0.5]{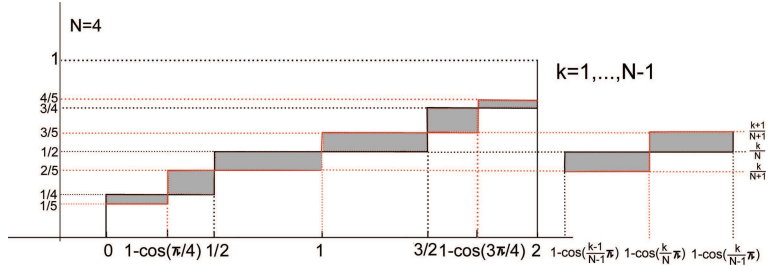}
\caption{An example of two neighboring paths with $N=4$ and $N+1=5$.}
\label{F:4.4}
\end{figure}
\begin{eqnarray*}
d_{1}(G,G')&=&\sum\limits_{k=1}^{N-1}\bigg\{\left(\cos{\frac{k-1}{N-1}\pi}-\cos{\frac{k}{N}\pi}\right)\left(\frac{k}{N}-\frac{k}{N+1}\right)\\
& &+\left(\cos{\frac{k}{N}-\cos{\frac{k}{N-1}\pi}\pi}\right)\left(\frac{k+1}{N+1}-\frac{k}{N}\right)\bigg\}\\
%&=&\frac{1}{N(N+1)}\sum\limits_{k=1}^{N-1}\bigg\{ k\left(\cos{\frac{k-1}{N-1}\pi}-\cos{\frac{k}{N}\pi}\right)+(N-k)\left(\cos{\frac{k}{N}\pi}-\cos{\frac{k}{N-1}\pi}\right) \bigg\}\\
&=&\frac{2}{N(N+1)}\sum\limits_{k=1}^{N-1}k\left(\cos{\frac{k-1}{N-1}\pi}-\cos{\frac{k}{N}\pi}\right)\hspace{.294\textwidth}\\
&=&\frac{1}{N(N+1)}\left(\cot^2{\frac{\pi}{2N}}-\cot^2{\frac{\pi}{2(N-1)}}+1\right).
\end{eqnarray*}
For the last equality above we use the Lagrange's trigonometric identities
\begin{equation*}\sum_{k=1}^N\sin kx=\frac{\cos\frac{1}{2}x-\cos(n+\frac{1}{2})x}{2\sin\frac{1}{2}x},\,\,
\sum_{k=1}^N\cos kx=\frac{\sin(n+\frac{1}{2})x-\sin\frac{1}{2}x}{2\sin\frac{1}{2}x},
\end{equation*}
and their derivatives.
\end{proof}

\begin{remark}
 By a Taylor expansion argument, we observe that
$$\cot^2{\frac{\pi}{2N}}-\cot^2{\frac{\pi}{2(N-1)}}=\mathcal{O}(N), \,\,\text{ as }N\rightarrow +\infty.$$
Therefore in this example, we have $d_1(G, G')=\mathcal{O}(1/N)$ as $N$ tends to infinity.
\end{remark}

We can calculate the example of cycles similarly.
\begin{pro}\label{p:cycles}
For two cycles
$G$ and $G'$ of size $N$ and $N+1$ respectively, we have
\begin{equation*}
d_1(G, G')=\left\{
             \begin{array}{ll}
               \frac{1}{N}+\frac{1}{N(N+1)}\left(\frac{1}{1-\cos(\frac{\pi}{N+1})}-\frac{4}{1-\cos(\frac{2\pi}{N})}\right), & \hbox{if $N$ is even;} \\
               \frac{1}{N+1}-\frac{1}{N(N+1)}\left(\frac{1}{1-\cos(\frac{\pi}{N})}-\frac{4}{1-\cos(\frac{2\pi}{N+1})}\right), & \hbox{if $N$ is odd.}
             \end{array}
           \right.
\end{equation*}
%
% is equal to
%$$, when $N$ is even, and
%$\frac{1}{N}+\frac{1}{N(N+1)}(\frac{2}{1-\cos(\frac{2\pi}{N+1})}-\frac{1}{2(1-\cos(\frac{\pi}{N}))})$,
%when $N$ is odd and $N>1$.
\end{pro}
\begin{remark}
 For $N$- and $(N+1)$-cycles, we also have $d_1(G, G')=\mathcal{O}(1/N)$ as $N$ tends to infinity.
\end{remark}

\section{Distance between large graphs}\label{s:Interlacing}
In this section we explore the behaviors of the spectral distance $d_1$ between large graphs in general. We require two large graphs are different from each other only by finite steps of operations which will be made clear in Remark \ref{graph
operations}. The main tool we employ is the so-called interlacing inequalities,
%We explore in this section the spectral distance between two
%connected finite weighted graphs $G$ and $G'$, one of which can be
%obtained from the other one by performing some standard operations.
%We will specify the operations in the following Remark \ref{graph
which describe the changes of the spectrum when we perform some operations on the underlying graph. Such kind of results for normalized Laplacian of a graph have been studied in \cite{chen2004interlacing, li2006short, butler2007interlacing, horak2013interlacing, AtayTuncel}. In fact, we can observe the interlacing phenomena of eigenvalues for paths and cycles in Proposition \ref{p:paths}  and \ref{p:cycles}.

Let the cardinality of vertices of $G$ and $G'$ be $N$ and $N-j$ respectively, where $j\in\mathbb{Z}$ can be either negative or positive.
Assume
$$0\leq\lambda_1\leq\lambda_2\leq\cdots\leq\lambda_N\text{ and }0\leq\lambda'_1\leq\lambda'_2\leq\cdots\leq\lambda'_{N-j}$$ are the spectra
of the corresponding normalized Laplacian $\Delta_G$ and $\Delta_{G'}$. Then interlacing inequalities have the following general form.
\begin{equation}\label{e:interlacing}
\lambda_{i-k_1}\leq\lambda'_i\leq\lambda_{i+k_2},\text{  for each }i=1,2,\ldots, N-j,
\end{equation}
with the notation that $\lambda_i=0$ for $i\leq 0$ and $\lambda_i=2$ for $i> N$, and $k_1, k_2$ are constants independent of the index $i$.

\begin{remark}\label{graph operations}
$G'$ can be obtained from $G$ by performing the following operations.
\begin{itemize}
  \item $G'$ is the proper difference of $G$ and one of its subgraph $L$. We say $L$ is a subgraph of $G$ if the weights $\theta_{L,uv}\leq \theta_{G,uv}$ for all $u, v$. And the proper difference of $G$ and $L$ is a weighted graph with weights $\theta_G-\theta_L$. In this case, $$k_1=\text{number of vertices in } L-\text{number of connected components of } L$$ and $$k_2=\text{number of vertices in } L$$ (Horak-Jost \cite[Corollary~2.11]{horak2013interlacing}, see also Butler \cite{butler2007interlacing}). This includes the operation of deleting an edge (see Chen et al \cite{chen2004interlacing} for the result for this particular operation). Symmetrically, this also covers the operation of adding a graph, see Bulter \cite{butler2007interlacing} for particular results and Atay-Tun\c{c}el \cite{AtayTuncel} for vertex replication.
  \item $G'$ is the image of an edge-preserving map $\varphi: G\rightarrow G'$. By an edge-preserving map here we mean an onto map from the vertices of $G$ to vertices of $G'$, such that $$\theta_{H,xy}=\sum_{\substack{u\in \varphi^{-1}(x)\\v\in\varphi^{-1}(y)}}\theta_{G,uv}$$ for all vertices $x,y$ of $G'$, and the degree of vertices are defined according to the edge weights as usual in both graphs. Notice that for our purpose, we do not allow $\varphi$ maps two neighboring vertices in $G$ to the same vertex in $G'$ in order to avoid self-loops. In this case, $$k_1=0\text{ and }k_2=j.$$ (Horak-Jost \cite[Theorem~3.8]{horak2013interlacing}.) This includes the operation of contracting vertices $u,v$ such that $N(u)\cap(N(v)\cup\{v\})=\emptyset$ (see Chen et al. \cite{chen2004interlacing}), where $N(u)$ stands for the neighborhood of $u$.
\item $G'$ is obtained from $G$ by contracting an edge. We only consider edges $uv$ in $G$ such that $d_u, d_v >1$. By edge contracting we mean deleting the edge $(u,v)$ and identifying $u$ and $v$ (Horak-Jost \cite[Definition~4.2]{horak2013interlacing}). Denote the number of common neighbors of $u, v$ by $m$. Then $$\text{when } m\neq 0, k_1=2m, k_2=1+2m;\text{ when }m=0, k_1=0, k_2=2.$$
(Horak-Jost \cite[Theorem~4.1]{horak2013interlacing}, where the unweighted normalized Laplacian case was discussed. We do not know whether it is also true
for weighted normalized Laplacian.)
%In this case, $G'$ can also be considered as the image of an edge-preserving map, but deleting the possible resulting self-loops.
\end{itemize}

\end{remark}
\begin{remark}
To the knowledge of the authors, the above three classes of operations includes all the operations discussed in the literature for interlacing results of normalized Laplacian.
\end{remark}

We prove the following result.
\begin{thm}\label{t:Asymptotic behavior}
Let $G$, $G'$ be two graphs, for which the spectra of corresponding normalized Laplacians satisfy (\ref{e:interlacing}). Then we have
\begin{equation}\label{e:large graphs behavior}
d_1(G,G')\leq C(k_1,k_2,j)\frac{1}{N}.
\end{equation}
\end{thm}
\begin{proof}
By definition, we have $$d_1(G, G')=d_1^W\left(\frac{1}{N}\sum_{i=1}^{N}\delta_{\lambda_i}, \frac{1}{N+j}\sum_{i+1}^{N+j}\delta_{\lambda'_i}\right).$$
By symmetry, w.l.o.g., we can suppose $j\geq 0$. We use a particular transport plan to derive the upper bound estimate. We move the mass $\frac{1}{N}$ from
$\lambda_i$ to $\lambda'_i$ for $i=1,2,\ldots, N-j$. We then move the mass at the remaining positions $\lambda_{N-j+1}, \ldots, \lambda_{N}$ to fill the gaps at $\lambda'_1, \lambda'_2, \ldots, \lambda'_{N-j}$ with a cost for every transportation at most $2$. That is, we have
\begin{align*}
d_1(G, G')&\leq \frac{1}{N}\sum_{i=1}^{N-j}|\lambda_i-\lambda'_i|+\frac{j}{N}\times 2\\
&\leq \frac{1}{N}\sum_{i=1}^{N+j}|\lambda_{i+k_2}-\lambda_{i-k_1}|+\frac{2j}{N}\\
&\leq \frac{k_1+k_2+1}{N}\sum_{i=1}^{N}|\lambda_{i+1}-\lambda_{i}|+\frac{2j}{N}\\
&\leq 2(k_1+k_2+j+1)\frac{1}{N}.
\end{align*}
In the second inequality above, we used interlacing inequalities (\ref{e:interlacing}). This complete the proof.
\end{proof}

\begin{remark}
 The disjoint union of a path of size $N$ and an isolated vertex can be obtained from a path of size $N+1$ by deleting an edge. A cycle of size $N$ can be obtained from a cycle of size $N+1$ by contracting an edge. Recall our calculation in Proposition \ref{p:paths} and \ref{p:cycles}, we see the estimate (\ref{e:large graphs behavior}) is sharp in the order of $1/N$.
\end{remark}
\begin{remark}
 This theorem tells that if two large graphs share similar structure, then the spectral distance between them is small.
\end{remark}

If $G'$ is the graph obtained from $G$ by performing operations such that $k_1, k_2$ are bounded (then $j$ is also bounded), we say $G'$ differs from $G$ by a bounded operation.
\begin{coro}
Let $\{G_i\}_{i=1}^{\infty}$ be a sequence of graphs with size $N_i$ tending to infinity. Assume that for any $i$, $G'_i$ differs from $G_i$ by a uniformly bounded operation, then $$\lim_{i\rightarrow\infty}d_1(G_i, G'_i)=0.$$
\end{coro}

\section{Applications to biological networks}\label{s:Biology}
%In this section, we apply the spectral distance $d_1$ to study the evolutionary history of networks in biological systems. In networks such as protein interaction network and genetic network, edge-rewiring and duplication-divergence are two edit operations which have proven to be closely related to some evolutionary mechanism, see \cite{ispolatov2005duplication, kim2012network}. In principle, the spectral distance reflects
% the structure differences of networks and intuitively describes the corresponding evolutionary process.
%As one of our interests is to explore relations between structure differences of a series of networks and the corresponding evolutionary process,
In real biological networks, such as protein interaction networks, edge-rewiring and duplication-divergence are two edit operations which have been proven to be closely related to some evolutionary mechanism, see \cite{ispolatov2005duplication, kim2012network}. For
a spectral analysis of the effect of such operations on protein interaction networks, we refer to \cite{BanJ07}. In this section, we apply the spectral distance $d_1$ to capture evolutionary signals in protein interaction networks through detecting their structural differences.
%In principle, the spectral distance reflects
We evolve graphs by operations of edge-rewiring and duplication-divergence, and then check the connection between the spectral distance $d_1$ and the evolutionary distance (i.e. the number of evolutionary operation steps). We restrict our simulations in the following to unweighted graphs.

Let us first explain the two edit operations on an unweighted graph $G=(V, E)$ explicitly.
\begin{itemize}
\item Edge-rewiring: Select randomly two edges $v_1v_3$, $v_4v_5\in E$ on four distinct vertices $v_1$, $v_3$, $v_4$, $v_5\in V$
(see Fig. \ref{F1:4_1_rewiring_dupliacation}(a)). Delete these two edges $v_1v_3$, $v_4v_5$ and add new edges $v_1v_4$, $v_3v_5$.
%Note that, after this operation, the size of the graph is preserved, and so is the degree sequence.
The size of the graph is preserved by this operation, and so is the degree sequence.
%Two edges with no sharing vertex are chosen randomly, i.e., for a graph $G=(V,E)$, $ij$, $kl\in E$ , and $i$, $j$, $k$, $l$ are four distinct vertices. Then, $ij$ and $kl$ are replaced with $ik$ and $jl$ (Fig. \ref{F1:4_1_rewiring_dupliacation}(a)). With this operation, the size of the graph is preserved, so is the degree sequence.
\item Duplication-divergence: Select randomly a target vertex $v_3\in V$. Add a replica $v_2$ of $v_3$ and new potential edges connecting $v_2$
with every neighbor of $v_3$. Each of these potential edges is activated with certain probability (0.5 in our simulations).
Then if at least one of these potential edges is established, keep the replica $v_2$; otherwise, delete the replica $v_2$ (see Fig. \ref{F1:4_1_rewiring_dupliacation}(b)).
%
%For a graph $G=(V,E)$, a randomly chosen target node is duplicated, that means its replica is introduced and connected to each neighbor of the target node. Each link emanating from the replica is activated with retention probability (0.5 in our simulations). If at least one link is established, the replica is preserved; otherwise the attempt is considered as failure and the graph is not changed. Fig. \ref{F1:4_1_rewiring_dupliacation}(b) shows the scheme of the duplication and divergence event. After several times of this operation, the size and degree sequence of the graph are changed with high probability.
\end{itemize}
\begin{figure}[!htb]
\centering
\includegraphics[width=3in]{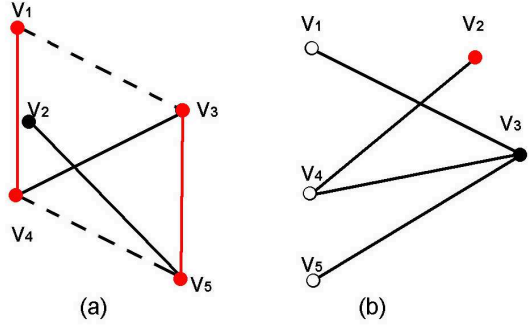}\vspace*{0pt}
\caption{(a) Edge-rewiring;
 %edges $v_{1}v_{3}$ and $v_{4}v_{5}$ are chosen randomly. They are removed, and new edges $v_{1}v_{4}$ and $v_{3}v_{5}$ are added;
(b) Duplication-divergence.
%$v_{3}$ is target node, and $v_{2}$ is its replica. Links $v_{1}v_{2}$ and $v_{2}v_{5}$ are disappeared as a result of divergence, and $v_{2}v_{4}$ is activated.
}
\label{F1:4_1_rewiring_dupliacation}
\end{figure}

Our simulations are designed as follows. We start form a Barab\'{a}si-Albert scale-free graph with $1000$ vertices. This is obtained through a mechanism incorporating growth and preferential attachment from a small complete graph of size $10$, see \cite{BA1999Science,barabasi2003scale}.
For each step of preferential attachment, we add one vertex with two edges.
%In fact, this kind of graphs is a very common type of large real networks, particularly of biological networks.
We remark that the Barab\'{a}si-Albert scale-free graph is not necessarily the best starting model for any biological network. However, it is closer to biological networks in many cases than
the other two popular models, the Erd\H{o}s-R\'{e}nyi random graph and the Watts-Strogatz small-world graph. Therefore, we use it as our starting point here.
%We carry out the edge-rewiring operation on this graph iteratively without any designed principle, i.e. all the correlative vertices and edges are chosen randomly. Then we plot the relation of the spectral distance and the evolutionary distance between new obtained graphs and the original scale-free one, see Fig. \ref{F1:4_1_relatonahip}. We also evolve this graph by duplication-divergence operations and investigate the relation between the two distances in that case correspondingly.
We carry out the operation of edge-rewiring and duplication-divergence on this graph iteratively, then plot the relationship of the spectral distance and the evolutionary distance between new obtained graphs and the original ones.

%
%We take these two operation on graphs without any designed principle, i.e., all the correlative vertices and edges are chosen randomly. Here we set the evolutionary distance by counting the number of evolutionary operation steps.
%To investigate the relationship between the evolutionary distance and the spectral distance, we start from
%one Barab\'{a}si-Albert scale-free graph with 1000 vertices (and two edges for each step of preferential
%attachments), which is a very common type of real networks \cite{barabasi2003scale}. With edge-rewiring and duplication-divergence, we can evolve a group of
%graphs with increasing evolutionary distance. Then, we plot the relation between the spectral distance and
%the evolutionary distance, see Fig. \ref{F1:4_1_relatonahip}.

In the plot of Fig. \ref{F1:4_1_relatonahip}, we observe that the spectral distance between graphs obtained by edge-rewiring operations and the original one increases more quickly than that obtained by duplication-divergence operations. This indicates that, after the same number of operation steps, edge-rewiring brings in more randomness to the graph than duplication-divergence. Recall also the fact that the sizes of graphs are invariant in the former case and vary in the later case.

\begin{figure}[h]
\centering
\includegraphics[width=3in]{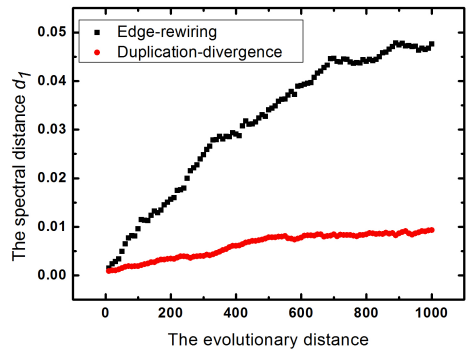}\vspace*{0pt}
\caption{The relation between the spectral distance $d_{1}$ and the evolutionary distance. Edit operations includes (a) edge-rewiring; (b) duplication-divergence.}
\label{F1:4_1_relatonahip}
\end{figure}

Although there is no strictly linear relation between the two distances, the spectral distance increases monotonically with respect to  the evolutionary distance. Based on this crucial point, the spectral distance is very useful for exploring the hiding evolutionary history of large real networks.

\section*{Acknowledgements}
The authors thank J\"urgen Jost for inspiring discussions on the topic of spectral distance of large networks. BH thanks Andreas Thom and Balint Vir\'ag for their patient explanations of random graphs.
SL thanks Willem Haemers for helpful discussions and pointing out the reference \cite{StevanovicProblems2007}. The authors thank the anonymous referees for their valuable comments which improved the presentation of the manuscript.

JG and SL were supported by the International Max Planck Research School "Mathematics in the Sciences". BH was partially supported from the funding of
the European Research Council under the European Union's Seventh
Framework Programme (FP7/2007-2013) / ERC grant agreement
n$^\circ$~267087. SL was partially supported by the EPSRC Grant EP/K016687/1 "Topology, Geometry and Laplacians of Simplicial Complexes".

%\bibliography{Graphdistances}

\begin{thebibliography}{99}
\bibitem{ATV} M. Ab\'ert, A. B. Thom, B. Vir\'ag, Benjamini-Schramm convergence and pointwise convergence of the spectral measure, preprint.
\bibitem{Aldous2007} D. Aldous, R. Lyons, Processes on unimodular random networks, Electron. J. Probab. 12 (2007), no. 54, 1454-1508.
\bibitem{AtayTuncel} F. M. Atay, H. Tun\c{c}el, On the spectrum of the normalized Laplacian for signed graphs: Interlacing, contraction, and replication,
Linear Algebra Appl. 442 (2014), 165-177.
\bibitem{BanJ07} A. Banerjee, J. Jost, Laplacian spectrum and protein-protein interaction networks, arXiv: 0705:3373, 2007.
\bibitem{BanJ08} A. Banerjee, J. Jost, Spectral plot properties: towards a qualitative classification of networks, Netw. Heterog. Media 3 (2008), no. 2, 395-411.
\bibitem{BanJLAA08} A. Banerjee, J. Jost, On the spectrum of the normalized graph Laplacian, Linear Algebra Appl. 428 (2008), no. 11-12, 3015-3022.
\bibitem{BanJ09} A. Banerjee, J. Jost, Graph spectra as a systematic tool in computational biology, Discrete Appl. Math. 157 (2009), no. 10, 2425-2431.
\bibitem{BA1999Science} A.-L. Barab{\'a}si and R. Albert, Emergence of scaling in random networks, Science 286 (1999), no. 5439, 509-512.
\bibitem{barabasi2003scale} A.-L. Barab{\'a}si, E. Bonabeau, Scale-Free networks, Scientific American, 2003.
\bibitem{BHJ12} F. Bauer, B. Hua, J. Jost, The dual Cheeger constant and spectra of infinite graphs, Adv. Math. 251 (2014), 147-194.
\bibitem{BJ} F. Bauer, J. Jost, Bipartite and neighborhood graphs and the spectrum of the normalized graph Laplacian, Comm. Anal. Geom. 21 (2013), no. 4, 787-845.
\bibitem{BJL12} F. Bauer, J. Jost, S. Liu, Ollivier-Ricci curvature and the spectrum of the normalized graph Laplace operator, Math. Res. Lett. 19 (2012), no. 6, 1185-1205.
\bibitem{Benjamini2001a} I. Benjamini, O. Schramm, Recurrence of distributional limits of finite planar graphs,
Electron. J. Probab. 6 (2001), no. 23, 13 pp.
\bibitem{biggs1994algebraic} N. Biggs, Algebraic graph theory, Second edition, Cambridge Mathematical Library,
Cambridge University Press, Cambridge, 1993.
%\bibitem{botti1993almost} P. Botti and R. Merris, Almost all trees share a complete set of immanantal polynomials. J. Graph Theory 17 (1993), no. 4, 467-476.
\bibitem{butler2007interlacing} S. Butler, Interlacing for weighted graphs using the normalized Laplacian. Electron. J. Linear Algebra 16 (2007), 90-98.
%\bibitem{bulter2011cospectral} S. Butler and J. Grout, A construction of cospectral graphs for the normalized Laplacian. Electron. J. Combin. 18 (2011), no. 1, Paper 231, 20 pp.
\bibitem{chen2004interlacing} G. Chen, G. Davis, F. Hall, Z. Li, K. Patel, M. Stewart, An interlacing result on normalized Laplacians, SIAM J. Discrete Math. 18 (2004), no. 2, 353-361.
\bibitem{chung1997spectral} F. R. K. Chung, Spectral graph theory, CBMS Regional
  Conference Series in Mathematics 92, American Mathematical Society,
  Providence, RI, 1997.
\bibitem{cvetkovic1980spectra} D. Cvetkovi\'{c}, M. Doob, H. Sachs, Spectra of graphs,
Theory and applications, Third edition, Johann Ambrosius Barth, Heidelberg, 1995.
\bibitem{FJ1984} A. M. Fink, M. Jr. Jodeit, On Chebyshev's other inequality, Inequalities in statistics and probability (Lincoln, Neb., 1982), 115-120,
IMS Lecture Notes Monogr. Ser., 5, Inst. Math. Statist., Hayward, CA, 1984.
\bibitem{Gu2014} J. Gu, The spectral distance based on the normalized Laplacian and applications to large networks, PhD Thesis, University of Leipzig, 2014.
\bibitem{GJS2014} J. Gu, J. Jost, S. Liu, P. F. Stadler, Spectral classes of regular, random, and empirical graphs, arXiv: 1406.6454, 2014.
\bibitem{HLP1934} G. H. Hardy, J. E. Littlewood, G. P\'{o}lya, Inequalities, 2d ed. Cambridge, at the University Press, 1952.
\bibitem{horak2013interlacing} D. Horak, J. Jost, Interlacing inequalities for eigenvalues of discrete Laplace operators, Ann. Global Anal. Geom. 43 (2013), no. 2, 177-207.
\bibitem{ispolatov2005duplication} I. Ispolatov, P. L. Krapivsky, A. Yuryev, Duplication-divergence model of protein interaction network, Phys. Rev. E 71 (6):061911, 2005.
\bibitem{JL11} J. Jost, S. Liu, Ollivier's Ricci curvature, local clustering and curvature dimension inequalities on graphs,
Discrete Comput. Geom. 51 (2014), no. 2, 300-322.
\bibitem{JSspectraldistance2012} I. Jovanovi\'{c}, Z. Stani\'{c}, Spectral distances of graphs, Linear Algebra Appl. 436 (2012), no. 5, 1425-1435.
\bibitem{KellerLenz12} M. Keller, D. Lenz, Dirichlet forms and stochastic completeness of graphs and subgraphs,
J. Reine Angew. Math. 666 (2012), 189-223.
\bibitem{kim2012network} J. Kim, I. Kim, S. K. Han, J. U. Bowie, S. Kim, Network rewiring is an important mechanism of gene essentiality change,
Scientific reports, 2 (2012).
\bibitem{Kuwae2003} K. Kuwae, T. Shioya, Convergence of spectral structures: a functional analytic theory and its applications to spectral geometry,
Comm. Anal. Geom. 11 (2003), no. 4, 599-673.
\bibitem{LOT13} J. R. Lee, S. Oveis Gharan, L. Trevisan, Multi-way spectral partitioning and higher-order Cheeger inequalities, STOC'12-Proceedings of the 2012 ACM Symposium on Theory of Computing, 1117-1130, ACM, New York, 2012; J. ACM 61 (2014), no. 6, 37:1-30.
\bibitem{li2006short} C.-K. Li, A short proof of interlacing inequalities on normalized Laplacians,
Linear Algebra Appl. 414 (2006), no. 2-3, 425-427.
\bibitem{Liu13} S. Liu, Multi-way dual Cheeger constants and spectral
  bounds of graphs, Adv. Math. 268 (2015), 306-338.
\bibitem{Lovasz} L. Lov\'{a}sz, Large networks and graph limits, American Mathematical Society Colloquium Publications, 60, American Mathematical Society, Providence, RI, 2012.
\bibitem{macindoe2010graph} O. Macindoe, W. Richards, Graph comparison using fine structure analysis,
SOCIALCOM '10 Proceedings of the 2010 IEEE Second International Conference on Social Computing,
Pages 193-200, IEEE Computer Society Washington, DC, USA, 2010.
\bibitem{Major1978} P. Major, On the invariance principle for sums of independent identically distributed random variables, J. Multivariate Anal. 8 (1978), no. 4, 487-517.
\bibitem{Mohar1989} B. Mohar, W. Woess, A survey on spectra of infinite graphs,
Bull. London Math. Soc. 21 (1989), no. 3, 209-234.
\bibitem{StevanovicProblems2007} D. Stevanovi\'{c}, Research problems from the Aveiro Workshop on graph spectra, Linear Algebra Appl. 423 (2007), no. 1, 172-181.
\bibitem{Thune12} M. Th\"{u}ne, Eigenvalues of matrices and graphs, PhD Thesis, University of Leipzig, 2012.
\bibitem{Trevisan2012} L. Trevisan, Max cut and the smallest eigenvalue, STOC'09-Proceedings of the 2009 ACM International Symposium on Theory of Computing, 263-271, ACM, New York, 2009;
SIAM J. Comput. 41 (2012), no. 6, 1769-1786.
\bibitem{Villani09} C. Villani, Topics in optimal transportation,
Graduate Studies in Mathematics, 58. American Mathematical Society, Providence, RI, 2003.
\end{thebibliography}
%\bibliographystyle{alpha}

\end{document}